\documentclass[11pt]{article}

\usepackage{amsmath}
\usepackage{amsthm}
\usepackage{amsfonts}
\usepackage{bbm}
\usepackage{amssymb}
\usepackage{graphicx}
\usepackage{color}

\usepackage{blindtext}
\usepackage{enumitem}

\usepackage{fancybox}
\newlength{\querylen}
\newcommand{\mmp}{\mathbb{P}}

\newcommand{\od}{\overset{{\rm d}}{=}}
\newcommand{\dod}{\overset{{\rm d}}{\to}}
\newcommand{\tp}{\overset{{\rm P}}{\to}}

\newcommand{\me}{\mathbb{E}}
\newcommand{\E}{\mathbb{E}}

\newcommand{\mr}{\mathbb{R}}
\newcommand{\mn}{\mathbb{N}}

\DeclareMathOperator{\1}{\mathbbm{1}}

\newtheorem{thm}{Theorem}[section]
\newtheorem{lemma}[thm]{Lemma}

\newtheorem{cor}[thm]{Corollary}

\newtheorem{assertion}[thm]{Proposition}
\newtheorem{corollary}[thm]{Corollary}

\theoremstyle{definition}

\theoremstyle{remark}
\newtheorem{rem}[thm]{Remark}

\begin{document}
\title{On nested infinite occupancy scheme in random environment}

\author{Alexander Gnedin\footnote{School of Mathematical Sciences, Queen Mary University of London,
Mile End Road, London E1 4NS, UK;\ e-mail: a.gnedin@qmul.ac.uk} \
and \ Alexander Iksanov\footnote{Faculty of Computer Science and
Cybernetics, Taras Shevchenko National University of Kyiv, 01601
Kyiv, Ukraine; \ e-mail: iksan@univ.kiev.ua}}

\maketitle
\begin{abstract}
\noindent We consider an infinite balls-in-boxes occupancy scheme
with boxes organised in nested hierarchy, and random probabilities
of boxes defined in terms of iterated fragmentation of a unit
mass. We obtain a multivariate functional limit theorem  for the
cumulative occupancy counts as the number of balls approaches
infinity. In the case of fragmentation  driven by a homogeneous
residual allocation model our result generalises the functional
central limit theorem for  the block counts in Ewens' and more
general regenerative partitions.
\end{abstract}


\hphantom{2000 Mathematics Subject Classification: }Secondary:
60C05
\section{Introduction}\label{Sect1}

In the  infinite multinomial occupancy scheme balls are thrown
independently in a series of boxes, so that each  ball  hits  box
$k=1,2,\dots$ with probability $p_k$, where $p_k>0$ and
$\sum_{k\in\mn}p_k=1$. This classical model is
sometimes named after Karlin due to his seminal contribution
\cite{Karlin:1967}. Features of the occupancy pattern emerging
after the first $n$ balls are thrown have been intensely studied,
see \cite{Barbour+Gnedin:2009, Gnedin+Hansen+Pitman:2007,
Iksanov:2016} for survey and references and
\cite{Hamou+Boucheron+Ohannesian:2017, Chebunin:2017+,
Chebunin+Kovalevskii:2016, Durieu+Samorodnitsky+Wang:2017+} for
recent advances. Statistics in focus of most of the previous work,
and also relevant to the subject of this paper, are not sensitive
to the labelling of boxes but rather only depend on the integer
partition of $n$ comprised of nonzero occupancy numbers.

In the infinite occupancy scheme in a random environment the
(hitting) probabilities of boxes are positive random variables $(P_k)_{k\in\mn}$ with
an arbitrary joint distribution satisfying $\sum_{k\in\mn}P_k=1$ almost surely.
Conditionally on $(P_k)_{k\in\mn}$, balls are thrown
independently, with probability $P_k$ of hitting box $k$.
Instances of this general setup have received considerable
attention within the circle of questions around exchangeable
partitions, discrete random measures and their applications to
population genetics, Bayesian statistics and computer science. In
the most studied and analytically best tractable case  the
probabilities of boxes  are representable as  the residual
allocation (or stick-breaking) model
\begin{equation}\label{BS}
P_k=U_1U_2\cdots U_{k-1}(1-U_k),\quad k\in\mn,
\end{equation}
where the $U_i$'s are independent with beta$(\theta,1)$
distribution{\footnote{Recall that a random
variable $X$ has a beta distribution with parameters $\alpha>0$
and $\beta>0$ if $\mmp\{X\in{\rm d}x\}=(1/{\rm B}(\alpha,
\beta))x^{\alpha-1}(1-x)^{\beta-1}\1_{(0,1)}(x){\rm d}x$. Here,
${\rm B}(\cdot, \cdot)$ is the beta function.} on $(0,1)$ and
$\theta>0$. In this case the distribution of the
sequence $(P_k)_{k\in\mn}$ is known as the
Griffiths-Engen-McCloskey (${\rm GEM}$) distribution
with parameter $\theta$. The sequence of the $P_k$'s arranged in decreasing
order has the Poisson-Dirichlet (${\rm PD}$) distribution
with parameter $\theta$, and the induced
exchangeable partition on the set of $n$ balls follows the
celebrated Ewens sampling formula
\cite{Arratia+Barbour+Tavare:2003, CSP, Pitman+Yakubovich:2017,
Pitman+Yakubovich:2017b}. Generalisations have been proposed in
various directions. The two-parameter extension due to Pitman and
Yor \cite{CSP} involves probabilities
of form \eqref{BS} with independent but not
identically distributed $U_i$'s, where the distribution of $U_i$
is beta$(\theta+\alpha i, 1-\alpha)$ (with $0<\alpha<1$ and
$\theta>-\alpha$). Residual allocation models with other choices
of parameters for the $U_i$'s with different beta distributions are found in \cite{Ishwaran+James,
Robert+Simatos:2009}. Much effort has been devoted to the
occupancy scheme, known as the Bernoulli sieve, which is based on
a {\it homogeneous} residual allocation model \eqref{BS}, that is,
with independent and identically distributed
(iid) factors $U_i$ having arbitrary distribution on $(0,1)$, see
\cite{Alsmeyer+Iksanov+Marynych:2017,Duchamps+Pitman+Tang:2017+,Gnedin+Iksanov+Marynych:2010b,
Iksanov:2016,Iksanov+Jedidi+Bouzeffour:2017,Pitman+Tang:2017+}.
The homogeneous model has a  multiplicative regenerative property,
also  inherited by the partition of the set of balls.

In more sophisticated constructions of random environments
probabilities $(P_k)_{k\in\mn}$ are identified with some
arrangement in sequence  of masses of a purely atomic  random
probability measure. A widely explored possibility is to define a
random cumulative distribution function $F$ by transforming  the
path of an increasing drift-free L{\'e}vy process (subordinator)
$(X(t))_{t\geq 0}$. In particular, in the Poisson-Kingman model
$F(t)= X(t)/X(1)$  for a measure supported by $[0,1]$, see
\cite{Ghosal, CSP}. In the regenerative model $F(t)=1-e^{-X(t)},
~t\geq 0,$   called in the statistical literature neutral-to-the
right prior \cite{Ghosal}, see
\cite{Barbour+Gnedin:2006,Gnedin+Iksanov:2012,Gnedin+Pitman+Yor:2006,Gnedin+Pitman+Yor:2006a}.

Following \cite{Bertoin:2008, Businger:2017, Joseph:2011} we shall
study a nested infinite occupancy scheme in random environment. In
this context we regard  $(P_k)_{k\in\mn}$  as a random {\it
fragmentation law} (with $P_k>0$ and $\sum_{k\in\mn}P_k=1$ a.s.).
To introduce hierarchy of boxes, for each $j\in\mn_0$ let
$\mathbb{V}_j$ be the set of words of length $j$ over $\mn$, where
$\mathbb{V}_0:=\{\varnothing\}$. The set
$\mathbb{V}=\bigcup_{j\in\mn_0} \mathbb{V}_j$ of all finite words
has the natural structure of a $\infty$-storey tree with root
$\varnothing$ and $\infty$-ary branching at every node, where $v1,
v2,\dots\in \mathbb{V}_{j+1}$ are the immediate followers of
$v\in\mathbb{V}_j$. Let $\{(P_k^{(v)})_{k\in\mn}$,
$v\in\mathbb{V}\}$ be a family of independent copies of
$(P_k)_{k\in\mn}$. With each $v\in\mathbb{V}$ we associate a box
divided in sub-boxes $v1, v2,\dots$ of the next level. The
probabilities of boxes are defined recursively by
\begin{equation}\label{fragm}
P(\varnothing)=1, ~~~P(vk)=P(v)P_k^{(v)}~~~{\rm for~}v\in
\mathbb{V}, k\in \mn
\end{equation}
(note that the factors $P(v)$ and $P_k^{(v)}$ are independent).
Given $(P(v))_{v\in\mathbb{V}}$, balls are thrown independently,
with probability $P(v)$ of hitting box $v$. Since
$\sum_{v\in\mathbb{V}_j}P(v)=1$ the allocation of balls in boxes
of level $j$ occurs according to the ordinary Karlin's occupancy
scheme.

Recursion \eqref{fragm} defines a discrete-time mass-fragmentation
process, where  the generic mass splits in proportions according
to the same  fragmentation law, independently of the history and
masses of the co-existing fragments. The nested occupancy scheme
can be seen as a combinatorial version of this fragmentation
process. Initially all balls are placed in box $\varnothing$, and
at each consecutive step $j+1$ each ball in box $v\in\mathbb{V}_j$
is placed in sub-box $vk$ with probability $P_k^{(v)}$. The
inclusion relation on the hierarchy of boxes induces a
combinatorial structure on the (labelled) set of balls called
total partition, that is a sequence of refinements from the
trivial one-block partition down to the partition in singletons.
The paper \cite{Forman+Haulk+Pitman} highlights the role of
exchangeability and gives  the general de Finetti-style connection
between mass-fragmentations and total partitions.

We consider the random probabilities of the hierarchy of boxes
and the outcome of throwing infinitely many balls all defined on
the same underlying probability space. For $j,r\in\mn$, denote by
$K_{n,j,r}$ the number of boxes $w\in{\mathbb W}_j$ of the $j$th
level that contain exactly $r$ out of $n$ first balls, and let
\begin{equation}\label{kn} K_{n,j}(s):=\sum_{r=\lceil n^{1-s}
\rceil}^n K_{n,j,r},\quad s\in [0,1],
\end{equation}
be a cumulative count of occupied boxes, where
$\lceil\,\cdot\,\rceil$ is the integer ceiling function. With
probability one the random function $s\mapsto K_{n,j}(s)$ is
nondecreasing and right-continuous, hence belongs to the Skorokhod
space $D[0,1]$. Also observe that $K_{n,j}(0)=K_{n,j,n}$ is zero
unless all balls fall in the same  box and that $K_{n,j}(1)$ is
the number of occupied boxes in the $j$th level. In
\cite{Bertoin:2008}  a central limit theorem with random centering
was proved for $K_{n,j}(1)$ for $j$ growing with $n$ at  certain
rate. Our focus is different. We are interested in the joint weak
convergence of $((K_{n,j_1}(s),\ldots,
K_{n,j_m}(s)))_{s\in[0,1]}$, properly normalised and centered, for
any finite collection of occupancy levels $1\leq j_1<\ldots< j_m$
as the number of balls $n$ tends to $\infty$. As far as we know,
this question has not been addressed so far. We prove a
multivariate functional limit theorem (Theorem \ref{main0})
applicable to the fragmentation laws representable by homogeneous
residual allocations models (including the ${\rm GEM}/{\rm PD}$
distribution) and some other models where the sequence of $P_k$'s
arranged in decreasing order approaches zero sufficiently fast. A
univariate functional limit  for $(K_{n,1}(s))_{s\in[0,1]}$ in the
case of Bernoulli sieve was previously obtained in
\cite{Alsmeyer+Iksanov+Marynych:2017}.

\section{Main result}\label{main345}

For given fragmentation law $(P_k)_{k\in\mn}$, let
$\rho(s):=\#\{k\in\mn: P_k\geq 1/s\}$ for $s>0$, and
$N(t):=\rho(e^t), V(t):=\me N(t)$ for $t\in\mr$. The joint
distribution of $K_{n,j,r}$'s is completely determined by the
probability law of the random function $\rho(\cdot)$, which
captures the fragmentation law up to re-arrangement of $P_k$'s.
For our purposes therefore we can make no difference between
fragmentation laws with the same  $\rho(\cdot)$.

Similarly, using probabilities of boxes in level $j\in\mn$ define
$\rho_j(s):=\#\{v\in{\mathbb V}_j: P(v)\geq 1/s\}$ for $s>0$, and
$N_j(t):=\rho_j(e^t), V_j(t):=\me N_j(t)$ for $t\in\mr$. Note that
$N_j(t)=0$ for $t\leq0$. Since $\sum_{v\in{\mathbb V}_j}P(v)=1$
a.s.\ we have $\rho_j(s)\leq s$, whence $N_j(t)\leq e^t$ a.s. and
$V_j(t)<e^t$.

Let $T_k:=-\log P_k$ for $k\in\mn$. Here is a basic decomposition
of principal importance for what follows:
\begin{equation}\label{basic decomp}
N_j(t)=\sum_{k\in\mn} N_{j-1}^{(k)}(t-T_k),\quad t\in\mr,
\end{equation}
where $(N_{j-1}^{(k)}(t))_{t\geq 0}$ for $k\in\mn$ are independent
copies of $(N_{j-1}(t))_{t\geq 0}$ which are also independent of
$T_1$, $T_2,\ldots$ An immediate consequence of \eqref{basic
decomp} is a recursion for the expectations
\begin{equation}\label{recexp}
V_j(t)=\int_{[0,\,t]}V_{j-1}(t-y){\rm d}V(y),\quad t\geq 0,~j\geq
2,
\end{equation}
which shows that $V_j(\cdot)$ is the $j$th convolution power of
$V(\cdot)$.


The assumptions on fragmentation law and the functional limit will
involve a centered Gaussian process $W:=(W(s))_{s\geq 0}$ which is
a.s.\ locally H\"{o}lder continuous with exponent $\beta>0$ and
satisfy $W(0)=0$. In particular, for any $T>0$
\begin{equation}\label{sam22}
|W(x)-W(y)|\leq M_T|x-y|^\beta,\quad 0\leq x,y\leq T
\end{equation}
for some a.s.\ finite random variable $M_T$. For each $u>0$, we
set further
$$R^{(u)}_1(s):=W(s),\quad R^{(u)}_j(s):=\int_{[0,\,s]}(s-y)^{u(j-1)}{\rm
d}W(y),\quad s\geq 0,~ j\geq 2.$$ For $j\geq 2$, the process
$R^{(u)}_j$ is understood as the result of integration by parts
$$R_j^{(u)}(s)=u(j-1)\int_0^s (s-y)^{u(j-1)-1}W(y){\rm d}y,\quad s\geq 0.$$ In particular, when $u(j-1)$ is a positive integer,
$$R^{(u)}_j(s)=(u(j-1))!\int_0^{s_1}\int_0^{s_2}\ldots\int_0^{s_{u(j-1)}}
W(y){\rm d}y{\rm d}s_{u(j-1)}\ldots{\rm d}s_2,\quad s\geq 0,~j\geq
2,$$ where $s_1=s$, which can be seen with the help of repeated
integration by parts.

Throughout the paper $D:=D[0,\infty)$ denotes the standard
Skorokhod space. Here is our main result.
\begin{thm}\label{main0}
Assume the following conditions hold:
\begin{itemize}
\item[\rm(i)]
\begin{equation}\label{1}
b_0+b_1 t^{\omega-\varepsilon_1}\leq V(t)-c t^\omega\leq
a_0+a_1t^{\omega-\varepsilon_2}
\end{equation}
for all $t\geq 0$ and some constants $c,\omega, a_0, a_1>0$, $0<\varepsilon_1,
\varepsilon_2\leq \omega$ and $b_0, b_1\in\mr$,
\item[\rm(ii)]
\begin{equation}\label{vari}
\me \sup_{s\in [0,\,t]}(N(s)-V(s))^2=O(t^{2\gamma}),\quad
t\to\infty
\end{equation}
for some $\gamma\in (\omega- \min (1,\varepsilon_1,
\varepsilon_2), \omega)$.
\item[\rm(iii)]
\begin{equation}\label{flt_assumption}
\frac{N(t\,\cdot)-c(t\,\cdot)^\omega}{at^\gamma}~\Rightarrow~
W(\cdot),\quad t\to\infty
\end{equation}
in the $J_1$-topology on $D$ for some $a>0$.
\end{itemize}
Then
\begin{equation}\label{relation_main5000}
\bigg(\frac{K_{n,j}(\cdot)- c_j(\log n(\cdot))^{\omega
j}}{ac_{j-1}(\log
n)^{\gamma+\omega(j-1)}}\bigg)_{j\in\mn}~\Rightarrow~(R^{(\omega)}_j(\cdot))_{j\in\mn},\quad
n\to\infty
\end{equation}
in the $J_1$-topology on $D[0,1]^\mn$, where $\Gamma(\cdot)$ is
the gamma function and
\begin{equation}\label{defcj}
c_j:=\frac{(c\Gamma(\omega+1))^j}{\Gamma(\omega j+1)},\quad j\geq 0.
\end{equation}
\end{thm}
\begin{rem}
The assumption $0<\varepsilon_1, \varepsilon_2\leq
\omega $ ensures that $\gamma>0$. Furthermore, in view of \eqref{1} and
the choice of $\gamma$ relation \eqref{flt_assumption} is
equivalent to
\begin{equation}\label{flt_assumption2}
\frac{N(t\,\cdot)-V(t\,\cdot)}{at^\gamma}~\Rightarrow~ W(\cdot),\quad
t\to\infty
\end{equation}
in the $J_1$-topology on $D$. Similarly, in view of \eqref{xyz}
given below relation \eqref{relation_main5000} is equivalent to
\begin{equation*}
\bigg(\frac{K_{n,j}(\cdot)- V_j(\log n(\cdot))}{ac_{j-1}(\log
n)^{\gamma+\omega(j-1)}}\bigg)_{j\in\mn}~\Rightarrow~(R^{(\omega)}_j(\cdot))_{j\in\mn},\quad
n\to\infty
\end{equation*}
in the $J_1$-topology on $D[0,1]^\mn$.
\end{rem}

\section{Proof of Theorem \ref{main0}}

\subsection{Auxiliary results}

\begin{lemma}\label{aux}

\begin{itemize}
\item[\rm(a)]
Condition \eqref{1} ensures that, for $j\in\mn$ and
$t\geq 0$,
\begin{equation}\label{xyz}
b_{0,j}+b_{1,j}t^{\omega    j-\varepsilon_1} \leq V_j(t)- c_j t^{\omega
j}\leq a_{0,j}+a_{1, j}t^{\omega j-\varepsilon_2},
\end{equation}
where $c_j$ is given by \eqref{defcj}, $a_{0,j}, a_{1,j}>0$ and
$b_{0,j}, b_{1,j}\in\mr$ are constants with $a_{0,1}:=a_0$,
$a_{1,1}:=a_1$, $b_{0,1}:=b_0$ and $b_{1,1}:=b_1$. In particular,
for $j\in\mn$,
\begin{equation}\label{asy100}
V_j(t)~\sim~c_j t^{\omega j},\quad t\to\infty
\end{equation}
and, for $j\in\mn$ and $u,v\geq 0$,
\begin{eqnarray}\label{est_imp}
V_j(u+v)-V_j(v)&\leq& c_j (\1_{\{\omega j \in (0,1]\}}u^{\omega
j}+\1_{\{\omega j>1\}}\omega j(u+v)^{\omega j-1}u)\notag
\\&+&a_{0,j}+a_{1,j}(u+v)^{\omega j-\varepsilon_2}-b_{0,j}-
b_{1,j}v^{\omega j-\varepsilon_1}.
\end{eqnarray}

\item[\rm(b)] Suppose \eqref{1} and \eqref{vari}. Then
\begin{equation}\label{slln22}
\lim_{t\to\infty}\frac{N(t)}{V(t)}=1\quad\text{{\rm a.s.}}
\end{equation}

\item[\rm(c)] Suppose \eqref{1} and \eqref{vari}. Then, for $j\in\mn$,
\begin{equation}\label{variest}
\me\sup_{s\in
[0,\,t]}(N_j(t)-V_j(t))^2=O(t^{2\gamma+2\omega(j-1)}),\quad
t\to\infty
\end{equation}

\end{itemize}
\end{lemma}
\begin{proof}
(a) We only prove the second inequality in \eqref{xyz}. To
this end, we first check that for any $b>0$
$$\int_{[0,\,t]}(t-y)^b{\rm d}V(y)\leq a_0 t^b+ba_1{\rm B}(b,1+\omega-\varepsilon) t^{\omega-\varepsilon+b}+bc{\rm B}(b,
1+\omega)t^{\omega+b},$$ where ${\rm B}(\cdot,\cdot)$ is the beta function, and
we write $\varepsilon$ for
$\varepsilon_2$ to ease notation. Indeed, using \eqref{1} we
obtain
\begin{eqnarray*}
\int_{[0,\,t]}(t-y)^b{\rm d}V(y)&=&b\int_0^t
(V(t-y)-c(t-y)^\omega)y^{b-1}{\rm d}y+bc\int_0^t (t-y)^\omega
y^{b-1}{\rm d}y\\&\leq& ba_0\int_0^t y^{b-1}{\rm d}y+ba_1\int_0^t
(t-y)^{\omega-\varepsilon}y^{b-1}{\rm d}y+bc\int_0^t (t-y)^\omega
y^{b-1}{\rm d}y\\&=&a_0 t^b+ba_1{\rm B}(b,1+\omega-\varepsilon)
t^{\omega-\varepsilon+b}+bc{\rm B}(b, 1+\omega)t^{\omega+b}.
\end{eqnarray*}

To prove the second inequality in \eqref{xyz}
we use  induction. The case $j=1$ is covered by \eqref{1}.
Assume the inequality holds for
$j=k-1$. Then, for $t\geq 0$ recalling (\ref{recexp}) we obtain
\begin{eqnarray*}
V_k(t)&=&\int_{[0,\,t]}(V_{k-1}(t-y)-c_{k-1}(t-y)^{\omega(k-1)}){\rm
d}V(y)+c_{k-1}\int_{[0,\,t]} (t-y)^{\omega(k-1)}{\rm d}V(y)\\&\leq&
a_{0,k-1}
V(t)+a_{1,k-1}\int_{[0,\,t]}(t-y)^{\omega(k-1)-\varepsilon}{\rm
d}V(y)+c_{k-1}\int_{[0,\,t]}(t-y)^{\omega(k-1)}{\rm d}V(y)\\&\leq&
a_{0, k-1}V(t)+a_{1, k-1}\big(a_0
t^{\omega(k-1)-\varepsilon}+(\omega(k-1)-\varepsilon) a_1{\rm
B}(\omega(k-1)-\varepsilon, 1+\omega-\varepsilon) t^{\omega
k-2\varepsilon}\\&+&(\omega(k-1)-\varepsilon) c {\rm
B}(\omega(k-1)-\varepsilon, 1+\omega)t^{\omega
k-\varepsilon}\big)\\&+&c_{k-1}\big(a_0
t^{\omega(k-1)}+\omega(k-1)a_1{\rm B}(\omega(k-1),1+\omega-\varepsilon)
t^{\omega k-\varepsilon}\\&+&\omega (k-1)c{\rm B}(\omega(k-1),
1+\omega)t^{\omega k}\big)\leq c_k t^{\omega k}+a_{0,k}+a_{1,k}t^{\omega
k-\varepsilon}
\end{eqnarray*}
for appropriate positive $a_{0,k}$ and $a_{1,k}$, where we used
\begin{equation}\label{equal}
c_k=c_{k-1}\omega(k-1)c{\rm B}(\omega(k-1), 1+\omega).
\end{equation}

Further, \eqref{asy100} is an immediate consequence of
\eqref{xyz}. To prove \eqref{est_imp}, we use \eqref{xyz} to
obtain, for $j\in\mn$ and $u,v\geq 0$,
\begin{equation*}
V_j(u+v)-V_j(v)\leq c_j ((u+v)^{\omega j}-v^{\omega j})+
a_{0,j}+a_{1,j}(u+v)^{\omega j-\varepsilon_2}-b_{0,j}-
b_{1,j}v^{\omega j-\varepsilon_1}.
\end{equation*}
If $\omega j\in (0,1]$, we have $(u+v)^{\omega j}-v^{\omega j}\leq
u^{\omega j}$ by subadditivity. If $\omega j>1$, we have
$(u+v)^{\omega j}-v^{\omega j}\leq \omega j(u+v)^{\omega j-1} u$
by the mean value theorem  and monotonicity. This completes the
proof of \eqref{est_imp}.

\vskip0.2cm \noindent (b) Condition \eqref{vari} ensures that ${\rm Var}\,N(t)=O(t^{2\gamma})$ as $t\to\infty$.  
Pick any $\delta>0$ such that $\delta(\omega-\gamma)>1/2$. An
application of Markov's inequality yields, for any $\varepsilon>0$
and positive integer $\ell$,
$$\mmp\{|N(\ell^\delta)-V(\ell^\delta)|>\varepsilon V(\ell^\delta)\}\leq \frac{{\rm Var}\,N(\ell^\delta)}{\varepsilon^2
V(\ell^\delta)^2}=O(\ell^{-2\delta (\omega-\gamma)}),\quad
\ell\to\infty.$$ This entails
$\lim_{\ell\to\infty}(N(\ell^\delta)/V(\ell^\delta))=1$ a.s.\ by
the Borel-Cantelli lemma. For any $t>1$ there exists an integer
$\ell\geq 2$ such that $(\ell-1)^\delta<t\leq \ell^\delta$,
whence, by monotonicity,
$$\frac{N((\ell-1)^\delta)}{V((\ell-1)^\delta)}\frac{V((\ell-1)^\delta)}{V(\ell^\delta)}\leq \frac{N(t)}{V(t)}\leq
\frac{N(\ell^\delta)}{V(\ell^\delta)}\frac{V(\ell^\delta)}{V((\ell-1)^\delta)}.$$
Since $\lim_{\ell\to\infty}(V(\ell^\delta)/V((\ell-1)^\delta))=1$
we infer \eqref{slln22}.

\vskip0.2cm \noindent (c) We use the induction on $j$. When $j=1$,
relation \eqref{variest} holds according to \eqref{vari}. Assuming
that \eqref{variest} holds for $j=i-1$ we intend to show that it
also holds for $j=i$.

Recalling \eqref{basic decomp}, write, for $i\geq 2$ and $t\geq
0$,
\begin{eqnarray}\label{decomp}
N_i(t)-V_i(t)&=&\sum_{k\in\mn}
\big(N^{(k)}_{i-1}(t-T_k)-V_{i-1}(t-T_k)\big)\\&+&
\bigg(\sum_{k\in\mn}
V_{i-1}(t-T_k)-V_i(t)\bigg)=:X_i(t)+Y_i(t).\notag
\end{eqnarray}
An integration by parts yields, for $s\geq 0$,
\begin{eqnarray*}
|Y_i(s)|&=&\Big|\int_{[0,\,s]}V_{i-1}(s-x){\rm
d}(N_1(x)-V_1(x))\Big|\leq \int_{[0,\,s]}|N_1(s-x)-V_1(s-x)|{\rm
d}V_{i-1}(x)\\&\leq& \sup_{y\in [0,\,s]}|N_1(y)-V_1(y)|
V_{i-1}(s).
\end{eqnarray*}
Hence,
$$\me [\sup_{s\in [0,\,t]}Y_i(s)]^2\leq \me [\sup_{y\in
[0,\,t]}(N(y)-V(y))]^2
V_{i-1}(t)^2=O(t^{2\gamma+2\omega(i-1)}),\quad t\to\infty$$ by
\eqref{vari} and \eqref{asy100}.

Passing to the analysis of $X_i$ we obtain, for $s\geq 0$
\begin{eqnarray*}
[\sup_{s\in[0,\,t]} X_i(s)]^2&\leq& \sup_{s\in [0,\,t]}\Big(
N_1(s)\sum_{k\in\mn}
\big(N^{(k)}_{i-1}(s-T_k)-V_{i-1}(s-T_k)\big)^2\1_{\{T_k\leq
s\}}\Big)\\&\leq& N_1(t)\sum_{k\in\mn}\sup_{s\in[0,\,t]}
\big(N^{(k)}_{i-1}(s)-V_{i-1}(s)\big)^2\1_{\{T_k\leq t\}}.
\end{eqnarray*}
Therefore, $\me [\sup_{s\in[0,\,t]} X_i(s)]^2\leq \me N(t)^2
\me[\sup_{s\in
[0,\,t]}(N_{i-1}(s)-V_{i-1}(s))]^2=O(t^{2\gamma+2\omega(i-1)})$ as
$t\to\infty$ by the induction assumption and the asymptotics $\me
[N(t)]^2={\rm Var}\,N(t)+V(t)^2\sim V(t)^2$ as $t\to\infty$. It
remains to note that
$$\me[\sup_{s\in[0,\,t]}(N_i(s)-V_i(s))]^2\leq 2\big(\me
[\sup_{s\in [0,\,t]}X_i(s)]^2+\me [\sup_{s\in
[0,\,t]}Y_i(s)]^2\big)=O(t^{2\gamma+2\omega(i-1)}), \quad
t\to\infty.$$
\end{proof}

Our main result, Theorem \ref{main0}, is an immediate consequence
of Proposition \ref{Propo3.3} given in Section \ref{box}, Theorem
\ref{main10} given next and its corollary.
\begin{thm}\label{main10}
Suppose \eqref{1}, \eqref{vari} and \eqref{flt_assumption}. Then
\begin{equation}\label{relation_main}
\bigg(\frac{N_j(t\cdot)-V_j(t\cdot)}{ac_{j-1}
t^{\gamma+\omega(j-1)}}\bigg)_{j\in\mn}~\Rightarrow~(R^{(\omega)}_j(\cdot))_{j\in\mn}
\end{equation}
in the $J_1$-topology on $D^\mn$.
\end{thm}
\begin{cor}\label{aux1}
Relation \eqref{relation_main} entails that, for $j\in\mn$ and
$h>0$,
\begin{equation}\label{conv_sup_j}
t^{-\gamma-\omega(j-1)}\sup_{y\in
[0,\,1]}(N_j(yt+h)-N_j(yt))~\tp~0,\quad t\to\infty.
\end{equation}
\end{cor}
It is convenient to prove Corollary \ref{aux1} at this early
stage.
\begin{proof}
Fix any $j\in\mn$. Since $R_j^{(\omega)}$ is a.s.\ continuous,
relation \eqref{relation_main} in combination with \eqref{xyz}
ensures that, for any $h>0$,
$$\Big(\frac{N_j(t\cdot)-c_j(t\cdot)^{\omega j}}{ac_{j-1}t^{\gamma+\omega(j-1)}},
\frac{N_j(t\cdot+h)-c_j(t\cdot+h)^{\omega
j}}{ac_{j-1}t^{\gamma+\omega(j-1)}}\Big)~\Rightarrow~
(R_j^{(\omega)}(\cdot), R_j^{(\omega)}(\cdot)),\quad t\to\infty$$
in the $J_1$-topology on $D\times D$, whence
$$t^{-\gamma-\omega(j-1)}\sup_{y\in [0,\,1]}(N_j(yt+h)-N_j(yt)-c_j((yt+h)^{\omega j}-(yt)^{\omega j}))~\tp~ 0,\quad t\to\infty.$$ Using
\begin{eqnarray*}
\sup_{y\in [0,\,1]}((yt+h)^{\omega j}-(yt)^{\omega j})&\leq&
\1_{\{\omega j \in (0,1]\}}h^{\omega j}+\1_{\{\omega j>1\}}\omega
j(t+h)^{\omega j-1}h
\end{eqnarray*}
we conclude that the right-hand side is
$o(t^{\gamma+\omega(j-1)})$ as $t\to\infty$ because
$\gamma>\omega-1$ by assumption. The proof is complete.
\end{proof}

Theorem \ref{main10} follows, in its turn, from Propositions
\ref{limit234} and \ref{zero}. Below we use the processes $X_j$
and $Y_j$ as defined in \eqref{decomp}.
\begin{assertion}\label{limit234}
Suppose \eqref{1} and \eqref{flt_assumption}. Then
\begin{equation}\label{limit10000}
\bigg(\frac{N_1(t \cdot)-V_1(t \cdot)}{at^\gamma}, \bigg(
\frac{Y_j(t\cdot)}{ac_{j-1}t^{\gamma+\omega(j-1)}}\bigg)_{j\geq
2}\bigg)~\Rightarrow~ (R^{(\omega)}_j(\cdot))_{j\in\mn},\quad
t\to\infty,
\end{equation}
in the $J_1$-topology on $D^\mn$.
\end{assertion}
\begin{assertion}\label{zero}
Suppose \eqref{1}, \eqref{vari} and \eqref{flt_assumption}. Then,
for each integer $j\geq 2$ and each $T>0$,
\begin{equation}\label{inter2023}
t^{-(\gamma+\omega(j-1))}\sup_{y\in [0,\,T]}X_j(ty)~\tp~0,\quad
t\to\infty.
\end{equation}
\end{assertion}

\subsection{Connecting two ways of box-counting}\label{box}

We retreat for a while from our main theme to focus on Karlin's occupancy scheme with deterministic probabilities $(p_k)_{k\in\mn}$.
By the law of large numbers a box of probability $p$ gets occupied by about $np$ balls, provided  $np$ is big enough.
This suggests to relate counting the boxes occupied by at least $n^{1-s}$ balls to the number of
boxes with probability at least $n^{-s}$.
Let $\bar \rho(t):=\#\{k\in\mn: p_k\geq 1/t\}$ for
$t>0$, and let  $\bar K_{n,r}$ be the number of boxes containing
exactly $r$ out of $n$ balls.
We shall estimate uniformly the difference between
$$\bar K_n(s):=\sum_{r=\lceil n^{1-s}\rceil}^n \bar K_{n,r}\,,\quad s\in [0,1],$$
and $(\bar \rho(n^s))_{s\in[0,1]}$. Proposition 4.1 in
\cite{Alsmeyer+Iksanov+Marynych:2017}. However, we did not succeed
to apply the cited proposition directly and will combine the
estimates obtained in its proof.
\begin{assertion}\label{prop1}
The following universal estimate holds for each $n\in\mn$
\begin{eqnarray}\label{appr}
\me \sup_{s\in [0,1]}\big|\bar K_n(s)-\bar \rho(n^s)\big|&\leq&
4\big(\bar \rho(n)-\bar \rho(y_0 n(\log n)^{-2})\big)+2\bar
\rho(n)(\log n)^{-1}\\&+&\int_1^\infty x^{-2}(\bar \rho(nx)-\bar
\rho(n)){\rm d}x+2\sup_{s\in [0,1]}(\bar \rho(en^s)-\bar
\rho(e^{-1}n^s)),\notag
\end{eqnarray}
where $y_0\in (0,1)$ is a constant which does not depend on $n$,
nor on $(p_k)_{k\in\mn}$.
\end{assertion}
\begin{proof}
For $k\in\mn$, denote by $\bar Z_{n,k}$ the number of balls falling  in the
$k$th box, so that
$$\bar K_n(s)=\sum_{k\in\mn}\1_{\{n^{1-s}\leq \bar Z_{n,k}\leq n\}},~s\in [0,1].$$
Then, for
$n\in\mn$ and $s\in [0,1]$,
\begin{eqnarray*}
|\bar K_n(1-s)-\bar \rho(n^{1-s})|&\leq& \sum_{{k\in\mn}}\1_{\{\bar
Z_{n,k}\geq n^s,\, 1\leq np_k\leq n^s\}}+\sum_{{k\in\mn}}\1_{\{\bar
Z_{n,k}\geq n^s,\, np_k<1\}}+\sum_{{k\in\mn}}\1_{\{\bar Z_{n,k}\leq
n^s,\, np_k\geq n^s\}}\\&:=&A_n(s)+B_n(s)+C_n(s).
\end{eqnarray*}
In \cite{Alsmeyer+Iksanov+Marynych:2017} it was shown that,
for $n\in\mn$, $$\me \sup_{s\in [0,1]}A_n(s)\leq 2(\bar \rho(n)-\bar \rho(y_0 n(\log n)^{-2}))+\frac{\bar \rho(n)}{\log n}+
\sup_{s\in [0,1]}(\bar \rho(en^s)-\bar \rho(n^s))$$ (see \cite{Alsmeyer+Iksanov+Marynych:2017}, pp.~1004--1005) and
$$\me \sup_{s\in [0,1]}C_n(s)\leq 2(\bar \rho(n)-\bar \rho(y_0 n(\log n)^{-2}))
+\frac{\bar \rho(n)}{\log n}+\sup_{s\in [0,1]}(\bar \rho(n^s)-\bar
\rho(e^{-1}n^s))$$ (see \cite{Alsmeyer+Iksanov+Marynych:2017}, p. 1006). Finally,
for $n\in\mn$,
\begin{eqnarray*}
\me \sup_{s\in [0,1]}B_n(s)&=&\me \sum_{k\in\mn}\1_{\{\bar
Z_{n,k}\geq 1,\,np_k<1\}}=\sum_{k\in\mn}
(1-(1-p_k)^n)\1_{\{np_k<1\}}\leq \sum_{k\in\mn}
np_k\1_{\{np_k<1\}}\\&=&\int_{(1,\infty)}\frac{1}{x}{\rm d}(\bar
\rho(nx)-\bar \rho(n))=\int_1^\infty \frac{\bar \rho(nx)-\bar
\rho(n)}{x^2} {\rm d}x.
\end{eqnarray*}
Combining the estimates we arrive at \eqref{appr} because
$$\sup_{s\in [0,1]}\big|\bar K_n(s)-\bar \rho(n^s)\big|=\sup_{s\in
[0,1]}\big|\bar K_n(1-s)-\bar \rho(n^{1-s})\big|.$$
\end{proof}

We apply next Proposition \ref{prop1} to the
setting of Theorem \ref{main0}. This result shows that
\eqref{relation_main5000} is equivalent to the analogous limit relation
with $\rho_j(n^t)=N_j(t \log n)$ replacing $K_{n,j}(t)$.
\begin{assertion}\label{Propo3.3}
Suppose \eqref{1} and \eqref{flt_assumption}. Then, for each
$j\in\mn$,
\begin{equation}\label{uniform_est}
\frac{\sup_{s\in [0,1]}\big|K_{n,j}(s)- \rho_j(n^s)|}{(\log
n)^{\gamma+\omega(j-1)}}~\tp~ 0, \quad n\to\infty.
\end{equation}
\end{assertion}
\begin{proof}
Fix any $j\in\mn$. By Proposition \ref{prop1},
for $n\in\mn$,
\begin{eqnarray}\label{asymp567}
&&\me \Big(\sup_{s\in
[0,1]}\big|K_{n,j}(s)-\rho_j(n^s)\big|\Big|(P_k)_{k\in\mn}\Big)\\&\leq&
4\big(\rho_j(n)-\rho_j(y_0 n(\log n)^{-2})\big)+2\rho_j(n)(\log
n)^{-1}\notag\\&+&\int_1^\infty x^{-2}(\rho_j(nx)-\rho_j(n)){\rm
d}x+2\sup_{s\in [0,1]}(\rho_j(en^s)-\rho_j(e^{-1}n^s)).\notag
\end{eqnarray}

Recall the notation
$$c_j=\frac{(c\Gamma(\omega+1))^j}{\Gamma(\omega j+1)},\quad j\in\mn$$ and our choice of $\gamma>\omega-\min
(1,\varepsilon_1,\varepsilon_2)$.  In view of \eqref{asy100},
\begin{equation}\label{i3}
\frac{\me \rho_j(n)}{\log n}=\frac{V_j(\log n)}{\log n}~\sim~
c_j(\log n)^{\omega j-1}=o((\log n)^{\gamma+\omega(j-1)}), \quad
n\to\infty.
\end{equation}
The next step is to show that
\begin{equation}\label{inter100}
\me \int_1^\infty x^{-2}(\rho_j(nx)-\rho_j(n)) {\rm d}x=o((\log
n)^{\gamma+\omega(j-1)}),\quad n\to\infty.
\end{equation}
As a preparation for the proof of \eqref{inter100} we first note
that according to \eqref{est_imp}
\begin{eqnarray*}
\me (\rho_j(nx)-\rho_j(n))&=& V_j(\log n+\log x)-V_j(\log
n)\notag\\&\leq& c_j (\1_{\{\omega j\in (0,1]\}}
(\log x)^{\omega j}+\1_{\{\omega j>1\}}\omega j (\log n+\log x)^{\omega j-1}\log
x)\notag\\&+&a_{0,j}+a_{1,j}(\log n+\log x)^{\omega
j-\varepsilon_2}-b_{0,j}+|b_{1,j}|(\log n)^{\omega j-\varepsilon_1}
\end{eqnarray*}
for $n\in\mn$ and $x\geq 1$. Further, using the inequality
$(u+v)^\alpha\leq (2^{\alpha-1}\wedge 1)(u^\alpha+v^\alpha)$ which
holds for $\alpha>0$ and $u,v\geq 0$ yields
$$\int_1^\infty x^{-2}(\log n+\log x)^{\omega j-\varepsilon_2}{\rm
d}x = O((\log n)^{\omega j-\varepsilon_2}) ,\quad n\to\infty$$ and
$$\int_1^\infty x^{-2}(\log n+\log x)^{\omega j-1}{\rm
d}x=O((\log n)^{\omega j-1}) ,\quad n\to\infty,$$ and
\eqref{inter100} follows.

An appeal to \eqref{xyz} enables us to conclude that for large enough
$n$
\begin{eqnarray*}
&&\me (\rho_j(n)-\me \rho_j(y_0 n (\log n)^{-2}))\\&=&V_j(\log
n)-V_j(\log n+\log y_0-2\log\log n)\\&\leq& c_j (\log n)^{\omega
j}\Big(1-\Big(1-\frac{2\log\log n-\log y_0}{\log n}\Big)^{\omega
j}\Big)\\&+&a_{0,j}+a_{1,j}(\log n)^{\omega
j-\varepsilon_2}-b_{0,j}-b_{1,j}(\log n+\log y_0-2\log\log
n)^{\omega j-\varepsilon_1}\\&\leq& 4\omega j c_j (\log n)^{\omega
j-1}\log\log n+a_{0,j}+a_{1,j}(\log n)^{\omega
j-\varepsilon_2}\\&-&b_{0,j}+|b_{1,j}|(\log n+\log y_0-2\log\log
n)^{\omega j-\varepsilon_1}.
\end{eqnarray*}
Hence,
\begin{equation}\label{i1}
\me (\rho_j(n)-\rho_j(y_0 n (\log n)^{-2}))=o((\log
n)^{\gamma+\omega(j-1)}),\quad n\to\infty
\end{equation}
by the same reasoning as above. Finally,
\begin{equation}\label{i2}
\frac{\sup_{s\in [0,1]}(\rho_j(en^s)-\rho_j(e^{-1}n^s))}{(\log
n)^{\gamma+\omega(j-1)}}=\frac{\sup_{s\in [0,1]}(N_j(s\log
n+1)-N_j(s\log n-1))}{(\log n)^{\gamma+\omega(j-1)}}~\tp~ 0,\quad
n\to\infty
\end{equation}
by Corollary \ref{aux1}. Using \eqref{i3}, \eqref{inter100},
\eqref{i1} and \eqref{i2} in combination with Markov's inequality
(applied to the first three terms on the right-hand side of
\eqref{asymp567}) shows that the left-hand side of
\eqref{asymp567} converges to zero in probability as $n\to\infty$.
Now \eqref{uniform_est} follows by another application of Markov's
inequality and the dominated convergence theorem.
\end{proof}

\subsection{Proof of Proposition \ref{limit234}}

We shall use an integral representation which has already appeared
in the proof of Lemma \ref{aux} (c):
\begin{equation}\label{yj12}
Y_j(t)=\sum_{k\in\mn}
V_{j-1}(t-T_k)-V_j(t)=\int_{[0,\,t]}V_{j-1}(t-y){\rm
d}(N_1(y)-V_1(y))
\end{equation}
for $j\geq 2$ and $t\geq 0$.

In view of \eqref{flt_assumption2} Skorokhod's representation
theorem ensures that there exist versions $\widehat{N}_1$ and
$\widehat{W}$ such that
\begin{equation}\label{cs2222}
\lim_{t\to\infty}\sup_{y\in [0,\,T]}\bigg|\frac{\widehat
N_1(ty)-V_1(ty)}{at^\gamma}-\widehat{W}(y)\bigg|=0\quad\text{a.s.}
\end{equation}
for all $T>0$. This implies that \eqref{limit10000} is equivalent
to
\begin{equation}\label{fd222}
\bigg(\widehat{W}(\cdot), \bigg(\frac{\widehat{Z}_j (t,
\cdot)}{c_{j-1}t^{\omega(j-1)}}\bigg)_{j\geq 2}\bigg)~
~\Rightarrow~ (R_j^{(\omega)}(\cdot))_{j\in\mn},\quad t\to\infty,
\end{equation}
where 
we set 
$\widehat{Z}_j(t,x):=\int_{(0,\,x]}\widehat{W}(y){\rm
d}_y(-V_{j-1}(t(x-y))$ for $j\geq 2$ and $t,x\geq
0$. As far as the first coordinate is concerned the equivalence is
an immediate consequence of \eqref{cs2222}. As for the other
coordinates, integration by parts yields, for $s>0$ fixed and
$j\geq 2$
\begin{eqnarray*}
\int_{[0,\,st]}\frac{V_{j-1}(st-y)}{c_{j-1}t^{\omega(j-1)}}{\rm
d}_y\frac{\widehat{N}_1(y)-V_1(y)}{at^\gamma}&=&\int_{(0,\,s]}\bigg(\frac{\widehat{N}_1(ty)-V_1(ty)}{at^\gamma}
-\widehat{W}(y)\bigg){\rm
d}_y\frac{-V_{j-1}(t(s-y))}{c_{j-1}t^{\omega(j-1)}}\\&+&\int_{(0,\,s]}\widehat{W}(y){\rm
d}_y\frac{-V_{j-1}(t(s-y))}{c_{j-1}t^{\omega(j-1)}}.
\end{eqnarray*}
Observe that the first term is a counterpart of \eqref{yj12} in
which $\widehat{N}_1$ replaces $N_1$. Denoting by $L(t)$ the first
term on the right-hand side, we infer
$$|L(t)|\leq \sup_{0\leq y\leq s}
\bigg|\frac{\widehat{N}_1(ty)-V_1(ty)}{at^\gamma}-\widehat{W}(y)\bigg|
\big(\big(c_{j-1} t^{\omega (j-1)}\big)^{-1}V_{j-1}(st)\big)\to 0~~{\rm a.s.~for~} t\to\infty$$
 in view of \eqref{asy100}
which implies that
\begin{equation}\label{aux111}
\lim_{t\to\infty} \big(c_{j-1}t^{\omega
(j-1)}\big)^{-1}V_{j-1}(st)=s^{\omega (j-1)}
\end{equation}
and \eqref{cs2222}.

For $j\geq 2$ and $t,x\geq 0$, set
$Z_j(t,x):=\int_{(0,\,x]}W(y){\rm d}_y(-V_{j-1}(t(x-y))$ and note
that \eqref{fd222} is equivalent to
\begin{equation}\label{fd333}
\bigg(W(\cdot), \bigg(\frac{Z_j(t,\cdot)}{c_{j-1}t^{\omega
(j-1)}}\bigg)_{j\geq 2}\bigg)~\Rightarrow~
(R_j^{(\omega)}(\cdot))_{j\in\mn},\quad t\to\infty
\end{equation}
because the left-hand sides of \eqref{fd222} and \eqref{fd333}
have the same distribution.

It remains to check two properties:

\noindent (a) weak convergence of finite-dimensional
distributions, i.e. that for all $n\in\mn$, all $0\leq
s_1<s_2<\ldots<s_n<\infty$ and all integer
$\ell\geq 2$
\begin{equation}\label{fd111}
\bigg(W(s_i), \bigg(\frac{Z_j (t,s_i)}{c_{j-1}t^{\omega
(j-1)}}\bigg)_{2\leq j\leq \ell}\bigg)_{1\leq i\leq
n}~\overset{{\rm d}}{\to}~ (R^{(\omega)}_j(s_i))_{1\leq j\leq
\ell,\,1\leq i\leq n}
\end{equation}
as $t\to\infty$;

\noindent (b) tightness of the distributions of coordinates in
\eqref{fd333}, excluding the first one.

\noindent {\sc Proof of \eqref{fd111}}. If $s_1=0$, we have
$W(s_1)=Z_j(t,s_1)=R_k^{(\omega)}(s_1)=0$ a.s.\ for $j\geq 2$ and
$k\in\mn$. Hence, in what follows we consider the case $s_1>0$.
Both the limit and the converging vectors in \eqref{fd111} are
Gaussian. In view of this it suffices to prove that
\begin{eqnarray}\label{cova222}
&&\lim_{t\to\infty} t^{-\omega (k+j-2)} \E
[Z_k(t,s)Z_j(t,u)]=c_{k-1}c_{j-1}\E
[R^{(\omega)}_k(s)R^{(\omega)}_j(u)]\\&=&
\begin{cases}
        c_{k-1}c_{j-1}\int_0^s\int_0^u r(s-y, u-z){\rm
d}y^{\omega(k-1)}{\rm d}z^{\omega(j-1)}, & \text{if} \ k,j\geq 2,   \\
        c_{j-1}\int_0^u r(s, u-z){\rm d}z^{\omega(j-1)}, & \text{if} \ k=1, j\geq 2
\end{cases}
\notag
\end{eqnarray}
for $k,j\in\mn$, $k+j\geq 3$ and $s,u>0$, where we set
$Z_1(t,\cdot)=W(\cdot)$ and $r(x,y):=\me [W(x)W(y)]$ for $x,y\geq
0$. We only consider the case where $k,j\geq 2$, the complementary
case being similar and simpler.

To prove \eqref{cova222} we need some preparation. For each $t>0$
denote by $\theta_{k,t}$ and $\theta_{j,t}$ independent random
variables with the distribution functions $\mmp\{\theta_{k,t}\leq
y\}=V_{k-1}(ty)/V_{k-1}(ts)$ on $[0,s]$ and
$\mmp\{\theta_{j,t}\leq y\}=V_{j-1}(ty)/V_{j-1}(tu)$ on $[0,u]$,
respectively. Further, let $\theta_k$ and $\theta_j$ denote
independent random variables with the distribution functions
$\mmp\{\theta_k\leq y\}=(y/s)^{\omega(k-1)}$ on $[0,s]$ and
$\mmp\{\theta_j\leq y\}=(y/u)^{\omega(j-1)}$ on $[0,u]$,
respectively. According to \eqref{asy100}, $(\theta_{k,t},
\theta_{j,t})\dod (\theta_k, \theta_j)$ as $t\to\infty$. Now
observe that the function $r(x,y)=\me [W(x)W(y)]$ is continuous,
hence bounded, on $[0,T]\times [0,T]$ for every $T>0$. This
follows from the assumed a.s.\ continuity of $W$, the dominated
convergence theorem in combination with $\me [\sup_{z\in
[0,\,T]}W(z)]^2<\infty$ for every $T>0$ (for the latter, see
Theorem 3.2 on p.~63 in \cite{Adler:1990}). As a result,
$r(s-\theta_{k,t}, u-\theta_{j,t})\dod r(s-\theta_k, u-\theta_j)$
as $t\to\infty$ and thereupon $$\lim_{t\to\infty} \me
r(s-\theta_{k,t}, u-\theta_{j,t})=\me r(s-\theta_k, u- \theta_j)$$
by the dominated convergence theorem.

This together with \eqref{aux111} leads to formula
\eqref{cova222}:
\begin{eqnarray*}
&&\me
[t^{-\omega(k+j-2)}Z_k(t,s)Z_j(t,u)]\\&=&\frac{V_{k-1}(ts)}{t^{\omega(k-1)}}\frac{V_{j-1}(tu)}{t^{\omega(j-1)}}
\int_0^s\int_0^u r(s-y, u-z){\rm
d}_y\Big(\frac{V_{k-1}(ty)}{V_{k-1}(ts)}\Big){\rm
d}_z\Big(\frac{V_{j-1}(tz)}{V_{j-1}(tu)}\Big)\\&=&\frac{V_{k-1}(ts)}{t^{\omega(k-1)}}\frac{V_{j-1}(tu)}{t^{\omega(j-1)}}\me
r(s-\theta_{k,t}, u-\theta_{j,t})~\to~c_{k-1}s^{k-1}c_{j-1}s^{j-1}
\me r(s-\theta_k, u-\theta_j)\\&=&c_{k-1}c_{j-1}\int_0^s\int_0^u
r(s-y, u-z){\rm d}y^{\omega(k-1)}{\rm d}z^{\omega(j-1)}
\end{eqnarray*}
as $t\to\infty$.

\noindent {\sc Proof of tightness}. Choose $j\geq 2$. We intend to
prove tightness of $(t^{-\omega(j-1)}Z_j (t,u))_{u\geq 0}$ on
$D[0,T]$ for all $T>0$. Since the function $t\mapsto
t^{-\omega(j-1)}$ is regularly varying at $\infty$ it is enough to
investigate the case $T=1$ only. By Theorem 15.5 in
\cite{Billingsley:1968} it suffices to show that for any
$\kappa_1>0$ and $\kappa_2>0$ there exist $t_0>0$ and $\delta>0$
such that
\begin{equation}\label{tight222}
\mmp\big\{\sup_{0\leq u,v\leq 1, |u-v|\leq \delta}|Z_j(t,u)-Z_j(t,
v)|>\kappa_1 t^{-\omega(j-1)}\big\}\leq \kappa_2
\end{equation}
for all $t\geq t_0$. We only analyze the case where $0\leq v<u\leq
1$, the complementary case being analogous.

Set $W(x)=0$ for $x<0$. The basic observation for the subsequent
proof is that \eqref{sam22} extends to
\begin{equation}\label{sam23}
|W(x)-W(y)|\leq M_T|x-y|^\beta
\end{equation}
whenever $-\infty<x,y\leq T$ for the same positive random variable
$M_T$ as in \eqref{sam22}. This is trivial when $x\vee y\leq 0$
and a consequence of \eqref{sam22} when $x\wedge y\geq 0$. Assume
that $x\wedge y\leq 0<x\vee y$. Then $|W(x)-W(y)|=|W(x\vee y)|\leq
M_T (x\vee y)^\beta\leq M_T|x-y|^\beta$, where the first
inequality follows from \eqref{sam22} with $y=0$.

Let $0\leq v<u\leq 1$ and $u-v\leq \delta$ for some $\delta\in
(0,1]$. Using \eqref{sam23} and \eqref{asy100} we obtain
\begin{eqnarray*}
t^{-\omega(j-1)}|Z_j(t,
u)-Z_j(t,v)|&=&t^{-\omega(j-1)}\bigg|\int_{[0,\, u)}
\big(W(u-y)-W(v-y)\big){\rm d}V_{j-1}(ty)\bigg|\\
\\&\leq& M_1(u-v)^\beta (t^{-\omega(j-1)}V_{j-1}(t))\leq
M_1\delta^\beta \lambda
\end{eqnarray*}
for large enough $t$ and a positive constant $\lambda$. This
proves \eqref{tight222}.

\subsection{Proof of Proposition \ref{zero}}

Relation \eqref{inter2023} will be proved by induction in two
steps.

\noindent {\sc Step 1}. 
Assume that \eqref{inter2023} holds for $j=2,\ldots, k$. We claim
that then
\begin{equation}\label{relation_main100}
\bigg(\frac{N_j(t\cdot)-V_j(t\cdot)}{ac_{j-1}t^{\gamma+\omega(j-1)}}\bigg)_{j=1,\ldots,k}~\Rightarrow~(R^{(\omega)}_j(\cdot))_{j=1,\ldots
k}
\end{equation}
in the $J_1$-topology on $D^k$. Indeed, in view of \eqref{decomp}
and the induction hypothesis relation \eqref{relation_main100} is
equivalent to
\begin{equation}\label{limit100}
\bigg(\frac{N_1(t \cdot)-V_1(t \cdot)}{at^\gamma}, \bigg(
\frac{Y_j(t\cdot)}{ac_{j-1}t^{\gamma+\omega(j-1)}}\bigg)_{j=2,\ldots,
k}\bigg)~\Rightarrow~ (R^{(\omega)}_j(\cdot))_{j=1,\ldots,
k},\quad t\to\infty.
\end{equation}
The latter holds by Proposition \ref{limit234}.

\noindent {\sc Step 2}. Using
\begin{equation}\label{relation_main101}
\frac{N_k(t\cdot)-V_k(t\cdot)}{ac_{k-1}t^{\gamma+\omega(k-1)}}~\Rightarrow~R^{(\omega)}_k(\cdot),\quad
t\to\infty
\end{equation}
in the $J_1$-topology on $D$ which is a consequence of
\eqref{relation_main100} we shall prove that \eqref{inter2023}
holds with $j=k+1$.

In view of \eqref{relation_main101} and the fact that
$R^{(\omega)}_k$ is a.s.\ continuous Skorokhod's representation
theorem ensures that there exists a probability space which is
rich enough to accomodate

\begin{itemize}

\item independent random processes $\widehat{N}^{(1)}_k$,
$\widehat{N}^{(2)}_k,\ldots$ which are versions of $N_k$;

\item independent random processes $\widehat{R}^{(\omega, 1)}_k$,
$\widehat{R}^{(\omega, 2)}_k,\ldots$ which are versions of
$R^{(\omega)}_k$;

\item random variables $\widehat T_1$, $\widehat
T_2,\ldots$ which are versions of $T_1$, $T_2,\ldots$ independent
of $(\widehat{N}^{(1)}_k, \widehat{R}_k^{(\omega, 1)})$,
$(\widehat{N}^{(2)}_k,\widehat{R}_k^{(\omega, 2)}),\ldots$

\end{itemize}

Furthermore,
\begin{equation}\label{cs22222}
\lim_{t\to\infty}\sup_{y\in [0,\,T]}\bigg|\frac{\widehat
N^{(r)}_k(ty)-V_k(ty)}{ac_{k-1}t^{\gamma+\omega(k-1)}}-\widehat{R}_k^{(\omega,
r)}(y)\bigg|=0\quad\text{a.s.}
\end{equation}
for all $T>0$ and $r\in\mn$.

Set $$\widehat X_{k+1}(t):=\sum_{r\in\mn}\big(\widehat
N^{(r)}_k(t-\widehat T_r)-V_k(t-\widehat T_r)\big)\1_{\{\widehat
T_r\leq t\}},\quad t\geq 0.$$ The process $(\widehat
X_{k+1}(t))_{t\geq 0}$ has the same distribution as
$(X_{k+1}(t))_{t\geq 0}$. Therefore, \eqref{inter2023} with
$j=k+1$ is equivalent to
\begin{equation}\label{prw_to_zero_as1111}
t^{-(\gamma+\omega k)}\sup_{y\in [0,\,T]}\widehat
X_{k+1}(ty)~\tp~0,\quad t\to\infty.
\end{equation}

To prove this, write
\begin{eqnarray*}
\frac{\widehat X_{k+1}(ty)}{ac_{k-1}t^{\gamma+\omega
k}}&=&t^{-\omega}\sum_{r\in\mn} \Big(\frac{\widehat
N^{(r)}_k(ty-\widehat T_r)-V_k(ty-\widehat
T_r)}{ac_{k-1}t^{\gamma+\omega(k-1)}}-\widehat R_k^{(\omega,
r)}(y-t^{-1}\widehat T_r)\Big)\1_{\{\widehat T_r\leq
ty\}}\\&+&t^{-\omega} \sum_{r\in\mn}\widehat R_k^{(\omega,
r)}(y-t^{-1}\widehat T_r)\1_{\{\widehat T_r\leq
ty\}}=:t^{-\omega}(\widehat Z_1(t,y)+\widehat Z_2(t,y)).
\end{eqnarray*}
For all $T>0$,
\begin{equation}\label{xxy}
t^{-\omega}\sup_{y\in [0,\,T]}|\widehat Z_1(t,y)|
\leq \sum_{r\in\mn}\sup_{y\in [0,\,T]}\Big|\frac{\widehat
N^{(r)}_k(ty)-V_k(ty)}{ac_{k-1}t^{\gamma+\omega(k-1)}}-\widehat
R_k^{(\omega, r)}(y)\Big|\1_{\{\widehat T_r\leq tT\}}.
\end{equation}
For $r\in\mn$, the random variables $\eta_r(t):=\sup_{y\in
[0,\,T]}\Big|\frac{\widehat
N^{(r)}_k(ty)-V_k(ty)}{ac_{k-1}t^{\gamma+\omega(k-1)}}-\widehat
R_k^{(\omega, r)}(y)\Big|$ are i.i.d.\ and independent of
$\widehat T_1$, $\widehat T_2,\ldots$. Furthermore,
\begin{equation}\label{ineq12}
\me [\eta_1(t)]^2\leq 2\Big(\frac{\me [\sup_{s\in
[0,\,Tt]}(N_k(s)-V_k(s))]^2}{(ac_{k-1}t^{\gamma+\omega(k-1)})^2}+\me
[\sup_{s\in [0,\,T]}R_k^{(\omega)}(s)]^2 \Big)=O(1),\quad
t\to\infty
\end{equation}
in view of \eqref{variest} and the well-known fact that the
supremum over $[0,T]$ of any a.s.\ continuous Gaussian process has
exponential tails. Since $\lim_{t\to\infty}\eta_1(t)=0$ a.s.,
inequality \eqref{ineq12} ensures that
$\lim_{t\to\infty}\E\eta_1(t)=0$. The right-hand side in
\eqref{xxy} is dominated by
$$t^{-\omega}\sum_{r\in\mn}(\eta_r(t)-\E\eta_r(t))\1_{\{\widehat T_r\leq t\}}+ t^{-\omega}\widehat
N(t)\me \eta_1(t),$$ where $\widehat{N}(t):=\#\{r\in\mn: \widehat
T_r\leq t\}$. Using the last limit relation and \eqref{slln22} we
conclude that the second summand converges to $0$ a.s., as
$t\to\infty$. The first summand converges to zero in probability,
as $t\to\infty$, by Markov's inequality in combination with
$t^{-2\omega}\me
\Big(\sum_{r\in\mn}(\eta_r(t)-\E\eta_r(t))\1_{\{\widehat T_r\leq
t\}}\Big)^2=t^{-2\omega}V(t) {\rm Var}\,\eta_1(t)=O(t^{-\omega})$.
Thus, for all $T>0$, $t^{-\omega}\sup_{y\in [0,\,T]}|\widehat Z_1(t,y)|\tp 0$ as $t\to\infty$.  


The process $(\widehat Z_2(t,y))$ has the same distribution as the
process $(Z_2(t,y))$ in which the random variables involved do not
carry the hats. Thus, it suffices to prove that
\begin{equation}\label{inter21}
t^{-\omega}\sup_{y\in [0,\,T]}|Z_2(t,y)|\tp 0,\quad t\to\infty.
\end{equation}
In what follows we write $\me_{(T_r)}(\cdot)$ for $\me
(\cdot|(T_r))$ and $\mmp_{(T_r)}(\cdot)$ for $\mmp (\cdot|(T_r))$.
Since $\me_{(T_r)} [Z_2(t,y)]^2=\me
[R_k^{(\omega)}(1)]^2\int_{[0,\,ty]}(y-t^{-1}x)^{2(\gamma+\omega(k-1))}{\rm
d}N_1(x)\leq \me [R_k^{(\omega)}(1)]^2
y^{2(\gamma+\omega(k-1))}N_1(ty)$. Using the Cram\'{e}r-Wold
device and Markov's inequality in combination with \eqref{slln22}
we infer that, given $(T_r)$, with probability one
finite-dimensional distributions of $(t^{-\omega}Z_2(t,y))_{y\geq
0}$ converge weakly to the zero vector, as $t\to\infty$. Thus,
\eqref{inter21} follows if we can show that the family of
$\mmp_{(T_r)}$-distributions of $(t^{-\omega} Z_2(t,y))_{y\geq 0}$
is tight. As a preparation, we observe that the process
$R_k^{(\omega)}$ inherits the local H\"{o}lder continuity of $W$.
Indeed, recalling \eqref{sam23} we obtain, for $x,y\in[0,T]$ and
$k\geq 2$,
\begin{equation}\label{hol}
|R_k^{(\omega)}(x)-R_k^{(\omega)}(y)|\leq \omega
(k-1)\int_0^{x\vee y}|W(x-z)-W(y-z)|z^{\omega(k-1)-1}{\rm d}z\leq
M_T T^{\omega(k-1)}|x-y|^\beta~\text{a.s.}
\end{equation}
It is also important that the random variable $M_T$ has finite
moments of all positive orders, see Theorem 1 in \cite{Azmoodeh
etal:2014}. Pick now integer $n\geq 2$ such that $2n\beta>1$. By
Rosenthal's inequality (Theorem 3 in \cite{Rosenthal:1970}), for
$x,y\in [0,\,T]$ and a positive constant $C$,
\begin{eqnarray*}
\me_{(T_r)} (Z_2(t,x)-Z_2(t,y))^{2n}&\leq&
C\Big(\Big(\sum_{r\in\mn}\me_{(T_r)}(R_k^{(\omega,
r)}(x-t^{-1}T_r)-R_k^{(\omega, r)}(y-t^{-1}T_r))^2 \1_{\{T_r\leq
tT\}}\Big)^n\\&+& \sum_{r\in\mn}\me_{(T_r)}(R_k^{(\omega,
r)}(x-t^{-1}T_r)-R_k^{(\omega, r)}(y-t^{-1}T_r))^{2n}
\1_{\{T_r\leq tT\}}\Big)\\&\leq& 2C T^{2n\omega(k-1)} \me M_T^{2n}
|x-y|^{2n\beta}N_1(tT)^n
\end{eqnarray*}
having utilized \eqref{hol} for the second inequality. In view of
\eqref{slln22}, this entails that a classical sufficient condition
for tightness (formula (12.51) on p.~95 in
\cite{Billingsley:1968}) holds
$$t^{-2n\omega}\me_{(T_r)} (Z_2(t,x)-Z_2(t,y))^{2n}\leq
C_1(x-y)^{2n\beta}$$ for a positive constant $C_1$ and large
enough $t$. Thus, we have proved that \eqref{inter21} holds
conditionally on $(T_r)$, hence, also unconditionally. The proof
of Proposition \ref{zero} is complete.

\section{The case of homogeneous residual allocation model}


In this section we apply Theorem \ref{main0} to the case of fragmentation law given by
homogeneous residual allocation model \eqref{BS}. Let $B:=(B(s))_{s\geq 0}$ be a standard Brownian
motion (BM) and for $q\geq 0$ let
$$B_q(s):=\int_{[0,\,s]}(s-y)^q {\rm d}B(y), ~~s\geq 0.$$
The process $B_q:=(B_q(s))_{s\geq 0}$ is a centered Gaussian
process called the fractionally integrated BM or the
Riemann-Liouville process. Clearly $B=B_0$, and for $q\in\mn$ the
process can be obtained as a repeated integral of the BM. It is
known that $B_q$ is locally H{\"o}lder continuous with any
exponent $\beta<q+1/2$ \cite{Hu+Nualart+Song}.
\begin{thm}\label{main}
Let $(P_k)_{k\in\mn}$ be given by \eqref{BS} with iid $U_i$'s such that
$$\mu:=\me|\log U_1|<\infty, ~\sigma^2:={\rm Var}( \log U_1)\in (0,\infty)$$
and $\me |\log (1-U_1)|<\infty$. Then
$$\bigg(\frac{(j-1)!(K_{n,j}(\cdot)-(j!)^{-1}(\mu^{-1}\log n(\cdot))^j)}{\sqrt{\sigma^2\mu^{-2j-1}(\log n)^{2j-1}}}\bigg)_{j\in\mn}~\Rightarrow~
(B_{j-1}(\cdot))_{j\in\mn},\quad n\to\infty$$ in the product
$J_1$-topology on $D[0,1]^\mn$.
\end{thm}
\begin{proof}
Let $(\xi_k, \eta_k)_{k\in\mn}$ be independent copies of a random
vector $(\xi, \eta)$ with positive arbitrarily dependent
components. Denote by $(S_k)_{k\in\mn_0}$ the zero-delayed ordinary random walk with increments $\xi_k$, that
is, $S_0:=0$ and $S_k:=\xi_1+\ldots+\xi_k$ for $k\in\mn$. Consider
a {\it perturbed} random walk
\begin{equation}\label{perturbed}
\tilde T_k:=S_{k-1}+\eta_k,\quad k\in \mn
\end{equation}
and then define $\tilde N(t):=\#\{k\in\mn:
\tilde T_k\leq t\}$ and
$\tilde V(t):=\me \tilde N(t)$ for $t\geq 0$. It is clear that
\begin{equation}\label{equ}
\tilde V(t)=\me U((t-\eta)^+)=\int_{[0,\,t]}U(t-y){\rm d}\tilde
G(y),\quad t\geq 0
\end{equation}
where, for $t\geq 0$, $U(t):=\sum_{k\geq 0}\mmp\{S_k\leq t\}$ is
the renewal function and $\tilde G(t)=\mmp\{\eta\leq t\}$.

For  $P_k$ written as \eqref{BS},
$T_k=-\log P_k$ becomes
$$T_k=|\log U_1|+\ldots +|\log U_{k-1}|+|\log (1-U_k)|,\quad
k\in\mn$$ which is a particular case of \eqref{perturbed} with
$(\xi, \eta)=(|\log U_1|, |\log (1-U_1)|)$. In view of this and
Lemma \ref{appl_BS} given below, the conditions of Theorem
\ref{main0} hold with $\omega=\varepsilon_1=\varepsilon_2=1$,
$\gamma=1/2$, $c=\mu^{-1}$, $W=B$ and $R_j=B_{j-1}$.
\end{proof}

\begin{lemma}\label{appl_BS}
Assume that ${\tt m}:=\me\xi<\infty$, ${\tt s}^2:={\rm Var}\,
\xi\in (0,\infty)$ and $\me\eta<\infty$. Then
\begin{itemize}
\item[\rm(a) ]
\begin{equation}\label{est2}
b_1 \leq \tilde V(t)-{\tt m}^{-1}t\leq a_0,\quad t\geq 0
\end{equation}
for some constants $b_1<0$ and $a_0>0$.
Also,
\begin{equation*}
\frac{\tilde N(t\cdot)-{\tt m}^{-1}(t\cdot)}{({\tt s}^2 {\tt
m}^{-3} t)^{1/2}}~\Rightarrow~ B(\cdot),\quad t\to\infty
\end{equation*}
in the $J_1$-topology on $D$.
\item[\rm (b)] $\me \sup_{s\in [0,\,t]}(\tilde N(s)-\tilde V(s))^2=O(t)$ as $t\to\infty$.  
\end{itemize}
\end{lemma}
\begin{proof}

\noindent (a) A standard result of the renewal theory tells us
that
\begin{equation}\label{lord}
0\leq U(t)-{\tt m}^{-1} t\leq a_0,
\end{equation}
where $a_0$ is a known positive constant.
The second inequality in combination with $\tilde{V}(t)\leq
U(t)$ proves the second  inequality in \eqref{est2}. Using the
first inequality in \eqref{lord} yields
\begin{eqnarray*}
\tilde V(t)-{\tt m}^{-1}t&=&\int_{[0,\,t]}(U(t-y)-{\tt
m}^{-1}(t-y)){\rm d}\tilde G(y)\\&-& {\tt m}^{-1} \int_0^t
(1-\tilde G(y)){\rm d}y\geq -{\tt m}^{-1} \int_0^t (1-\tilde
G(y)){\rm d}y\geq -{\tt m}^{-1}\me\eta.
\end{eqnarray*}
For a proof of  weak convergence see Theorem 3.2 in
\cite{Alsmeyer+Iksanov+Marynych:2017}.

\noindent (b) We shall use a decomposition $$\tilde N(t)-\tilde
V(t)=\sum_{r\geq 0}(\1_{\{S_r+\eta_{r+1}\leq t\}}-\tilde
G(t-S_r))+\int_{[0,\,t]}\tilde G(t-x){\rm d}(\nu(x)-U(x)),$$ where
$\nu(x):=\#\{r\in\mn_0:S_r\leq t\}$ for $x\geq 0$, so that
$U(x)=\me\nu(x)$. It suffices to prove that
\begin{equation}\label{11}
\me\sup_{s\in [0,\,t]}\Big(\sum_{r\geq 0}(\1_{\{S_r+\eta_{r+1}\leq
s\}}-\tilde G(s-S_r))\Big)^2=O(t),\quad t\to\infty
\end{equation}
and
\begin{equation}\label{12}
D(t):=\me\sup_{s\in [0,\,t]}\Big(\int_{[0,\,s]}\tilde G(s-x){\rm
d}(\nu(x)-U(x))\Big)^2=O(t),\quad t\to\infty.
\end{equation}

\noindent {\sc Proof of \eqref{11}}. 
For each $j\in\mn$, we write
\begin{eqnarray*}
\sup_{s\in [j,\, j+1)} \sum_{r\geq 0}(\1_{\{S_r+\eta_{r+1}\leq
s\}}-\tilde G(s-S_r))&\leq& \sum_{r\geq
0}(\1_{\{S_r+\eta_{r+1}\leq j+1\}}-\tilde G(j+1-S_r))\\&+&
\sum_{r\geq 0}(\tilde G(j+1-S_r)-\tilde G(j-S_r)).
\end{eqnarray*}
Similarly,
\begin{eqnarray*}
\sup_{s\in [j,\, j+1)} \sum_{r\geq 0}(\1_{\{S_r+\eta_{r+1}\leq
s\}}-\tilde G(s-S_r))&\geq& \sum_{r\geq
0}(\1_{\{S_r+\eta_{r+1}\leq j\}}-\tilde G(j-S_r))\\&-& \sum_{r\geq
0}(\tilde G(j+1-S_r)-\tilde G(j-S_r)).
\end{eqnarray*}
Thus, \eqref{11} is a consequence of
\begin{equation}\label{14}
\sum_{j=0}^{[t]+1}\me \Big(\sum_{r\geq 0}(\1_{\{S_r+\eta_{r+1}\leq
j\}}-\tilde G(j-S_r))\Big)^2=O(t),\quad t\to\infty
\end{equation}
and
\begin{equation}\label{15}
\sum_{j=0}^{[t]+1}\me \Big(\sum_{r\geq 0}(\tilde G(j+1-S_r)-\tilde
G(j-S_r))\Big)^2=O(t),\quad t\to\infty.
\end{equation}
The second moment in \eqref{14} is equal to $\int_{[0,\,j]}\tilde
G(j-x)(1-\tilde G(j-x)){\rm d}U(x)\leq \int_{[0,\,j]}(1-\tilde
G(j-x)){\rm d}U(x)$. In view of $\me\eta<\infty$, the function
$x\mapsto 1-\tilde G(x)$ is directly Riemann integrable on
$[0,\infty)$. According to Lemma 6.2.8 in \cite{Iksanov:2016} this
implies that the right-hand side of the last inequality is $O(1)$,
as $j\to\infty$, thereby proving \eqref{14}.

Further, set $K(j):=\int_{[0,\,j]}(\tilde G(j+1-x)-\tilde
G(j-x)){\rm d}\nu(x)$ for $j\in\mn_0$. Then $$\me \Big(\sum_{r\geq
0}(\tilde G(j+1-S_r)-\tilde G(j-S_r))\Big)^2\leq 2(\me K(j)^2+\me
(\nu(j+1)-\nu(j))^2)\leq 2(\me K(j)^2+\me \nu(1)^2),$$ where the
last inequality is a consequence of distributional subadditivity
of $\nu$, that is, $\mmp\{\nu(t+s)-\nu(s)>x\}\leq
\mmp\{\nu(t)>x\}$ for $t,s,x\geq 0$. Recall that $\nu(1)$ has
finite exponential moments, so that trivially $\me
\nu(1)^2<\infty$. Left with estimating $\me K(j)^2$ we infer
\begin{eqnarray*}
\me K(j)^2&=&\me \Big(\tilde G(j+1)-\tilde
G(j)+\sum_{k=0}^{j-1}\int_{[k,\,k+1)}\tilde G(j+1-x)-\tilde
G(j-x){\rm d}\nu(x)\Big)^2 \\&\leq& \me \Big(1
+\sum_{k=0}^{j-1}(\tilde G(j+1-k)-\tilde
G(j-k))(\nu(k+1)-\nu(k))\Big)^2\\&\leq& 2\Big(1+ (\tilde G(j)
+\tilde G(j+1)-\tilde G(1))^2\sum_{k=0}^{j-1}\frac{\tilde
G(j+1-k)-\tilde G(j-k)}{\tilde G(j)+\tilde G(j+1)-\tilde G(1)} \me
(\nu(k+1)-\nu(k))^2\Big)\\&\leq& 2(1+(\tilde G(j) +\tilde
G(j+1)-\tilde G(1))^2\me \nu(1)^2)\leq 2(1+4\me \nu(1)^2).
\end{eqnarray*}
Here, the second inequality is implied by convexity of $x\mapsto
x^2$. Combining the obtained estimates together we arrive at
\eqref{15}.

\noindent {\sc Proof of \eqref{12}}. Assuming that
\begin{equation}\label{13}
\me\sup_{s\in[0,\,t]}(\nu(s)-U(s))^2=O(t),
\end{equation}
integration by parts in \eqref{12} yields $$D(t)=\me\sup_{s\in
[0,\,t]}\Big(\int_{[0,\,s]}(\nu(s-x)-U(s-x)){\rm d}\tilde
G(x))\Big)^2\leq \tilde G(t)^2 \me\sup_{s\in
[0,\,t]}(\nu(s)-U(s))^2=O(t)$$ which proves \eqref{12}.

Passing to the proof of \eqref{13} we first observe that in view
of \eqref{lord} relation \eqref{13} is equivalent to
\begin{equation}\label{133}
\me\sup_{s\in[0,\,t]}(\nu(s)-{\tt m}^{-1}s)^2=O(t),\quad
t\to\infty.
\end{equation}
Since $s\mapsto \nu(s)-{\tt m}^{-1}s$ is a (random) piecewise
linear function with slope $-{\tt m}^{-1}$ having unit jumps at
times $S_0$, $S_1,\ldots$ we conclude that
\begin{eqnarray*}
\sup_{s\in [0,\,t]}(\nu(s)-{\tt m}^{-1}s)^2&\leq& \max
\big(\max_{0\leq k\leq \nu(t)}(k-{\tt m}^{-1}S_k)^2, \max_{0\leq
k\leq \nu(t)-1}(k+1-{\tt m}^{-1}S_k)^2\big)\\&\leq&
2\big(1+\max_{0\leq k\leq \nu(t)}(k-{\tt m}^{-1}S_k)^2\big).
\end{eqnarray*}
Applying Doob's inequality to the martingale $(S_{\nu(t)\wedge
n}-{\tt m}(\nu(t)\wedge n))_{n\in\mn_0}$ (this is a martingale
with respect to the filtration generated by the $\xi_k$ because
$\nu(t)$ is a stopping time with respect to the same filtration)
we obtain
\begin{eqnarray*}
\me \max_{0\leq k\leq\nu(t)\wedge n}(S_k-{\tt m}k)^2&=&\me
\max_{0\leq k\leq n}(S_{\nu(t)\wedge k}-{\tt m}(\nu(t)\wedge
k))^2\\&\leq& 4\me (S_{\nu(t)\wedge n}-{\tt m}(\nu(t)\wedge
n))^2=4{\tt s}^2 \me (\nu(t)\wedge n)
\end{eqnarray*}
for each $n\in\mn$. Here, the last equality is nothing else but
Wald's identity. An application of L\'{e}vy's monotone convergence
theorem yields
$$\me \max_{0\leq k\leq\nu(t)}(S_k-{\tt m} k)^2\leq 4{\tt s}^2
U(t).$$ In view of \eqref{lord} the right-hand side is $O(t)$, as
$t\to\infty$, and \eqref{133} follows.


\end{proof}


Recall that $(P_k)_{k\in \mn}$ follows the ${\rm GEM}$
distribution with parameter $\theta>0$ when $U_i$'s in \eqref{BS}
are beta$(\theta,1)$-distributed, in which case $\mu=\me|\log
U_1|=\theta^{-1}, \sigma^2={\rm Var}(\log U_1)=\theta^{-2}$ and
$\me |\log(1-U_1)|=\theta\sum_{n\geq
1}n^{-1}(n+\theta)^{-1}<\infty$. 

\begin{corollary}\label{main}
For $\theta>0$ let $(P_k)_{k\in\mn}$ be ${\rm GEM}$-distributed
with parameter $\theta$, or any random sequence
such that the sequence of $P_k$'s arranged in decreasing order
follows the ${\rm PD}$ distribution with parameter $\theta$. Then
\begin{equation}\label{limpd}
\bigg(\frac{(j-1)!(K_{n,j}(\cdot)-(j!)^{-1}(\theta \log
n(\cdot))^j)}{\sqrt{(\theta \log
n)^{2j-1}}}\bigg)_{j\in\mn}~\Rightarrow~
(B_{j-1}(\cdot))_{j\in\mn},\quad n\to\infty.
\end{equation}
in the product $J_1$-topology on $D[0,1]^\mn$.
\end{corollary}

\section{Some regenerative models}\label{levy_appl}

For $(X(t))_{t\geq 0}$ a drift-free subordinator with $X(0)=0$ and
a nonzero L{\'e}vy measure $\nu$ supported by $(0,\infty)$ let
$$\Delta X(t)=X(t)-X(t-), ~t\geq 0,$$
be the associated process of jumps. The process $\Delta X(\cdot)$ assumes nonzero values on a countable set, which is dense in case $\nu(0,\infty)=\infty$.
The transformed process (multiplicative subordinator)
$F(t)= 1-e^{-X(t)},\,t\geq 0,$ has the associated process of jumps
$$\Delta F(t)= e^{-X(t-)}(1-e^{-\Delta X(t)}), ~t\geq 0.$$
In this section we identify the fragmentation law
$(P_k)_{k\in\mn}$ with nonzero jumps $\Delta F(\cdot)$ arranged in
some order (for instance by decrease). Note that multiplying the
L{\'e}vy measure by a positive factor corresponds to a time-change
for $F$, hence does not affect the derived fragmentation law.

We shall assume that the L\'{e}vy measure $\nu$ is infinite and
has the right tail $\nu([x,\infty))$ satisfying
\begin{equation}\label{logar}
\beta_0+\beta_1 |\log x|^{q-r_2} \leq \nu ([x,\infty))-c_0|\log
x|^q \leq \alpha_0+\alpha_1 |\log x|^{q-r_1}
\end{equation}
for small enough  $x>0$ and some $q, c_0,  \alpha_0, \alpha_1>0$,
$1/2<r_1, r_2\leq q+1$ and $\beta_0, \beta_1<0$.
\begin{thm}\label{main2000}
Assume that \eqref{logar} holds with $q, r_1, r_2\geq 1$ and
$$m:=\me X(1)=\int_{[0,\infty)}x \nu({\rm d}x)<\infty,~~ s^2:={\rm
Var}\,X(1)=\int_{[0,\infty)}x^2\nu({\rm d}x)<\infty.$$ Then
$$\bigg(\frac{K_{n,j}(\cdot)-c_j^\ast(\log n(\cdot))^{(q+1)j}}   {q{\rm B}(q,(q+1)j-q)sm^{-3/2}c_{j-1}^\ast (\log n)^{(q+1)j-1/2}}\bigg)_{j\in\mn}~\Rightarrow~
(B_{(q+1)j-1}(\cdot))_{j\in\mn},\quad n\to\infty$$ in the product
$J_1$-topology on $D[0,1]^\mn$, where
$$c_j^\ast:=\frac{c_0\Gamma(q+2)}{m(q+1)\Gamma((q+1)j+1)},\quad j\in\mn.$$
\end{thm}

Theorem \ref{main2000} applies to the gamma subordinator with the
L\'{e}vy measure $$\nu({\rm d}x)=\theta
x^{-1}e^{-\lambda x}\1_{(0,\infty)}(x){\rm
d}x$$ and to the subordinator with $$\nu({\rm
d}x)=\theta(1-e^{-x})^{-1}e^{-\lambda x}\1_{(0,\infty)}(x){\rm
d}x,$$ where $\theta,\lambda>0$. In both cases $s^2<\infty$ and \eqref{logar} holds with
$c_0=q=r_1=r_2=1$.

Theorem \ref{main2000} is a consequence of Theorem \ref{main0}, the easily checked formula
$$\int_{[0,\,u]}(u-y)^\alpha {\rm d}B_q(y)=q{\rm B}(q,\alpha+1)\int_{[0,\,u]}(u-y)^{q+\alpha}{\rm d}B(y),\quad u\geq 0,~\alpha,q>0$$ which we use for
$\alpha=(q+1)(j-1)$ and the next lemma.
\begin{lemma}\label{subord}
Assume that \eqref{logar} holds and $s^2<\infty$. Then the following is true:
\begin{itemize}
\item[\rm (a)]
\begin{equation}\label{ineq3567}
b_0+b_1t^{q-\min(r_2-1, 0)}\leq V(t)- c_0(m(q+1))^{-1}t^{q+1}\leq
a_0+a_1 t^{q-\min(r_1-1, 0)},\quad t>0
\end{equation}
for some constants $a_0$, $a_1>0$ and $b_0$, $b_1\leq 0$, where
$m=\me X(1)<\infty,$

\item[\rm (b)] $$\frac{N(t\cdot)-c_0(m(q+1))^{-1}(t\cdot)^{q+1}}{s m^{-3/2}t^{q+1/2}}~\Rightarrow~ B_q(\cdot),\quad t\to\infty$$
in the $J_1$-topology on $D$.

\item[\rm (c)] Assume that $q, r_1, r_2\geq 1$. Then $\me \sup_{s\in [0,\,t]}(N(s)-V(s))^2=O(t^{2q+1})$ as $t\to\infty$.

\end{itemize}
\end{lemma}
\begin{proof}
(a) Set $f(x):=\nu ([-\log (1-e^{-x}), \infty))$ for $x\geq 0$.
Inequality \eqref{logar} in combination with $\lim_{x\to\infty}\nu
([x,\infty))=0$ entails
\begin{equation}\label{logar2}
\beta_0+\beta_1 x^{q-r_2} \leq f(x)-c_0x^q \leq \alpha_0+\alpha_1
x^{q-r_1}
\end{equation}
for all $x>0$ and some constants $\alpha_0$, $\alpha_1$, $\beta_0$
and $\beta_1$ which are not necessarily the same as in
\eqref{logar}.

Since $$N(t)=\sum \1_{\{X(s-)-\log (1-e^{-\Delta X(s)})\leq
t\}}=\sum \1_{\{\Delta X(s)\geq -\log (1-e^{-(t-X(s-))})\}},$$
where the summation extends to all $s>0$ with $\Delta X(s)>0$, we
conclude that $V(x)=\me N(x)=\int_{[0,\,x]}f(x-y){\rm
d}U^\ast(y)$, where $U^\ast(x):=\int_0^\infty \mmp\{X(t)\leq
x\}{\rm d}x=\me T(x)$ is the renewal function and
$T(x):=\inf\{t>0: X(t)>x\}$ for $x\geq 0$.

Similarly to \eqref{lord} we have
\begin{equation}\label{lord1}
0\leq U^\ast(t)-m^{-1} t\leq a^\ast_0,\quad t\geq 0,
\end{equation}
where $a^\ast_0$ is a known positive constant. Using this and
\eqref{logar2} we infer
\begin{eqnarray*}
V(t)-c_0(m(q+1))^{-1}t^{q+1}&=&\int_{[0,\,t]}(U^\ast(t-y)-m^{-1}(t-y)){\rm
d}f(y)\\&+&m^{-1}\int_0^t (f(y)-c_0y ^q) {\rm d}y\leq a^\ast_0
f(t)+m^{-1}\int_0^t (\alpha_0+\alpha_1 y^{q-r_1}){\rm d}y\\&\leq&
a_0(\alpha_0+\alpha_1 t^{q-r_1}+c_0t^q)+m^{-1}(\alpha_0
t+\alpha_1(q-r_1+1)^{-1}t^{q-r_1+1}).
\end{eqnarray*}
This proves the second inequality in \eqref{ineq3567}. Arguing
analogously we obtain
\begin{eqnarray*}
V(t)-c_0(m(q+1))^{-1}t^{q+1}&\geq & m^{-1}\int_0^t (f(y)-c_0y ^q)
{\rm d}y\geq  m^{-1}\int_0^t (\beta_0+\beta_1 y^{q-r_2}){\rm
d}y\\&\geq& m^{-1}(\beta_0 t+ \beta_1 (q-r_2+1)^{-1}t^{q+1-r_2}),
\end{eqnarray*}
thereby proving the first inequality in \eqref{ineq3567}.

\noindent (b) Write
\begin{eqnarray}
N(t)&=& \sum \big( \1_{\{\Delta X(s)\geq -\log
(1-e^{-(t-X(s-))})\}}-f(t-X(s-))\big)\1_{\{X(s-)\leq
t\}}\notag\\&+& \sum f(t-X(s-))\1_{\{X(s-)\leq
t\}}=:N_1(t)+N_2(t).\label{repr}
\end{eqnarray}
As a preparation for the proof of part (b) we intend to show that
\begin{equation}\label{11123}
\lim_{t\to\infty}t^{-q-1/2} N_1(t)=0\quad \text{a.s.},\quad
t\to\infty.
\end{equation}

\noindent {\sc Proof of \eqref{11123}}. To reduce technicalities
to a minimum we only consider the case $q>1$. Since $\me
[N_1(t)]^2= V(t)$ and $V(t)\sim c_0(m(q+1))^{-1}t^{q+1}$ as
$t\to\infty$ we conclude that
$$\lim_{\mn\ni \ell\to\infty}\ell^{(q+1/2)}N_1(\ell)=0\quad \text{a.s.}$$ by the Borel-Cantelli
lemma. For each $t\ge 0$, there exists $\ell\in\mn_0$ such that
$t\in [\ell, \ell+1)$. Now we use a.s.\ monotonicity of $N(t)$ and
$N_2(t)$ to obtain
\begin{eqnarray*}
&&(\ell+1)^{-(q+1/2)}(N_1(\ell)-(N_2(\ell+1)-N_2(\ell)))\leq
t^{-(q+1/2)}N_1(t)\\&\le&
\ell^{-(q+1/2)}(N_1(\ell+1)+N_2(\ell+1)-N_2(\ell))\quad
\text{a.s.}
\end{eqnarray*}
Thus, it remains to check that
$$\lim_{\ell\to\infty} \ell^{-(q+1/2)}(N_2(\ell+1)-N_2(\ell))=0,\quad \text{a.s.}$$
In view of \eqref{logar2}, $f$ satisfies a counterpart of
\eqref{est_imp}, whence
\begin{eqnarray}
N_2(\ell+1)-N_2(\ell)&=&\int_{[0,\,\ell]}
(f(\ell+1-y)-f(\ell-y)){\rm d}T(y)+\int_{(\ell,\,
\ell+1]}f(\ell+1-y){\rm d}T(y)\label{ineq23}\\&\leq& (c_0
(q-1)(\ell+1)^{q-1}+\alpha_0+\alpha_1
(\ell+1)^{q-r_1}-\beta_0+|\beta_1|\ell^{q-r_2}+f(1))T(\ell+1)\notag\\&=&O(\ell^{q+\max(1-r_1,
1-r_2, 0)})=o(\ell^{q+1/2})\notag
\end{eqnarray}
a.s.\ as $\ell\to\infty$. For the penultimate equality we have
used the strong law of large numbers for $T(y)$. The last equality
follows from $r_1, r_2>1/2$.

We are ready to prove part (b). We shall use representation
\eqref{repr}. Relation \eqref{11123} entails
\begin{equation}\label{inter2}
t^{-q-1/2}\sup_{y\in [0,\,T]}N_1(ty)~\tp~0,\quad t\to\infty.
\end{equation}
for each $T>0$. Thus, we are left with showing that
$$\frac{N_2(t\cdot)-c_0(m(q+1))^{-1}(t\cdot)^{q+1}}{s m^{-3/2}t^{q+1/2}}~\Rightarrow~ B_q(\cdot),\quad t\to\infty$$
in the $J_1$-topology on $D$. The proof of this is similar to that
of weak convergence of the $j$th coordinate, $j\geq 2$, in
\eqref{limit10000}. The only difference is that, instead of
\eqref{flt_assumption2}, we use
$$\frac{T(t\cdot)-m^{-1}(t\cdot)}{s
m^{-3/2}t^{1/2}}~\Rightarrow~ B(\cdot),\quad t\to\infty$$ in the
$J_1$-topology on $D$, where $B$ is a Brownian motion, see Theorem
2a in \cite{Bingham:1972}.

\noindent (c) Since the proof is analogous to that of Lemma
\ref{appl_BS}(b) we only give a sketch. In view of \eqref{repr} it
suffices to prove that, as $t\to\infty$,
\begin{equation}\label{1122}
\me\sup_{s\in [0,\,t]}\Big(\sum \big( \1_{\{\Delta X(v)\geq -\log
(1-e^{-(s-X(v-))})\}}-f(s-X(v-))\big)\1_{\{X(v-)\leq
s\}}\Big)^2=O(t^{2q+1})
\end{equation}
and
\begin{equation}\label{1222}
\me\sup_{s\in [0,\,t]}\Big(\int_{[0,\,s]}f(s-x){\rm
d}(T(x)-U^\ast(x))\Big)^2=O(t^{2q+1}).
\end{equation}

\noindent {\sc Proof of \eqref{1122}}. Arguing as in the proof of
Lemma \ref{appl_BS}(b) we conclude that \eqref{1122} is a
consequence of
\begin{equation}\label{1422}
\sum_{\ell=0}^{[t]+1}\me \Big(\sum \big( \1_{\{\Delta X(v)\geq
-\log (1-e^{-(\ell-X(v-))})\}}-f(\ell-X(v-))\big)\1_{\{X(v-)\leq
\ell\}}\Big)^2=O(t^{2q+1})
\end{equation}
and
\begin{equation}\label{1522}
\sum_{\ell=0}^{[t]+1}\me (N_2(\ell+1)-N_2(\ell))^2=O(t^{2q+1}).
\end{equation}
The second moment in \eqref{1422} is equal to
$V(\ell)=O(\ell^{q+1})$. This entails that the left-hand side of
\eqref{1422} is $O(t^{q+2})$, hence $O(t^{2q+1})$ because of the
assumption $q\geq 1$. Finally, since $r_1, r_2\geq 1$ by
assumption and $\me [T(\ell)]^2=O(\ell^2)$, inequality
\eqref{ineq23} entails $\me
(N_2(\ell+1)-N_2(\ell))^2=O(\ell^{2q})$ and thereupon
\eqref{1522}.

\noindent {\sc Proof of \eqref{1222}}. Set $\hat
\nu(x):=\inf\{k\in\mn: X(k)>x\}$ for $x\geq 0$. Since $T(x)\leq
\hat \nu(x)\leq T(x)+1$ a.s.\ and, according to \eqref{13}, $\me
\sup_{s\in [0,\,t]}(\hat \nu(s)-\me \hat\nu(s))^2=O(t)$ as
$t\to\infty$, we infer $\me \sup_{s\in
[0,\,t]}(T(s)-U^\ast(s))^2=O(t)$ as $t\to\infty$. With this at
hand, relation \eqref{1222} readily follows.

The proof of Lemma \ref{subord} is complete.
\end{proof}

\section{The Poisson-Kingman model}

Let $(X(t))_{t\geq 0}$ be a subordinator as in Section
\ref{levy_appl} with the only differences that the parameters in
\eqref{logar} satisfy $q\in (0,2)$, $q/2<r_1, r_2\leq q$ and that
we additionally assume
\begin{equation}\label{logLevy}
\int_{(1,\infty)}(\log x)^s\nu({\rm d}x)<\infty,
\end{equation}
where $s=2q$ when $q\in (0,3/2)$ and $s=\varepsilon+q/(2-q)$ for
some $\varepsilon>0$ when $q\in [3/2, 2)$.

The ranked sequence of jumps of the process  $(X(t)/X(1))_{t\in
[0,1]}$ can be represented as $P_j:=L_j/L$, where $L_1\geq
L_2\geq\cdots>0$ is the sequence of atoms of a non-homogeneous
Poisson random measure with mean measure $\nu$, and
$L:=\sum_{j\geq 1} L_j\stackrel{{\rm d}}{=}X(1)$. This is the
Poisson-Kingman construction \cite[Section 3]{Pitman:2003} of
probabilities $(P_j)_{j\in\mn}$, which we regard as fragmentation
law underlying a nested occupancy scheme.
\begin{thm}\label{main2001}
Assume that the function $x\mapsto \nu((x,\infty))$ is strictly
increasing and continuous on $[0,\infty)$. For the fragmentation
law as described above limit relation \eqref{relation_main5000}
holds with $\omega=q$, $\gamma=q/2$, $c=c_0$, $a=c_0^{1/2}$ and
$W(s):=B(s^q)$ for $s\geq 0$ being a time changed Brownian motion.
\end{thm}
Theorem \ref{main2001} is a consequence of Theorem \ref{main0} and
Lemma \ref{subord10} given next.
\begin{lemma}\label{subord10}
Under the assumptions of Theorem \ref{main2001} the following is
true:
\begin{itemize}
\item[\rm (a)]
\begin{equation}\label{ineq3}
\beta_2+\beta_3 t^{q-r_4}\leq V(t)- c_0 t^q \leq \alpha_2+\alpha_3
t^{q-r_3},\quad t>0
\end{equation}
for some constants $\alpha_2,\alpha_3>0$, $q\in (0,2)$, $q/2<r_3,
r_4\leq q$ and $\beta_2$, $\beta_3<0$.

\item[\rm (b)]
$$\me\sup_{s\in [0,\,t]}(N(s)-V(s))^2=O(t^q),\quad t\to\infty.$$
\item[\rm (c)]
$$\frac{N(t\cdot)-c_0(t\cdot)^q}{(c_0 t^q)^{1/2}}~\Rightarrow~ W(\cdot),\quad t\to\infty$$
in the $J_1$-topology on $D$, where $W(s)=B(s^q)$ for $s\geq 0$.
\end{itemize}
Condition \eqref{logLevy} is not needed for part (c).
\end{lemma}
\begin{proof}
For $t\in\mr$, set $\widehat N (t):=\#\{k\in\mn: L_k\geq e^{-t}\}$
so that $N(t)=\#\{k\in\mn: L_k/L\geq e^{-t}\}=\widehat N(t-\log
L)$. Note that $N(t)=0$ for $t<0$. Further, put
$m(t):=\nu((e^{-t},\infty))$ for $t\in\mr$ and note that $m$ is a
strictly increasing and continuous function with $m(-\infty)=0$.
In view of \eqref{logar} 
\begin{equation}\label{ineq}
\beta_0+\beta_1 t^{q-r_2} \leq m(t)-c_0 t^q \leq \alpha_0+\alpha_1
t^{q-r_1}
\end{equation}
for\footnote{Actually, \eqref{logar} only ensures that
\eqref{ineq} holds for large enough $t$. However, adjusting
$\alpha_i$ and $\beta_i$, $i=0,1$ properly one obtains
\eqref{ineq} for all $t\geq 0$. Of course, $\alpha_i$ and
$\beta_i$ in \eqref{ineq} are not necessarily the same as in
\eqref{logar}.} $t\geq 0$, where $\alpha_0,\alpha_1>0$, $q\in
(0,2)$, $q/2<r_1, r_2\leq q$ and $\beta_0, \beta_1<0$. Later, we
shall need the following consequences of \eqref{ineq}:
\begin{equation}\label{asym}
m(t)~\sim~ c_0 t^q,\quad t\to\infty
\end{equation}
and
\begin{equation}\label{dini}
\lim_{t\to\infty}\sup_{s\in
[0,\,s_0]}\Big|\frac{m(ts)}{c_0t^q}-s^q\Big|=0
\end{equation}
for all $s_0>0$. For the latter we have also used Dini's theorem.

The random process $(\widehat N (t))_{t\in\mr}$ is non-homogeneous
Poisson. In particular, $\widehat N (t)$ has a Poisson
distribution of mean $m(t)$. Let
$\mathcal{P}:=(\mathcal{P}(t))_{t\geq 0}$ denote a homogeneous
Poisson process of unit intensity. Throughout the proof we use the
representation $(\widehat
N(t))_{t\in\mr}=(\mathcal{P}(m(t))_{t\in\mr}$ which gives us a
transition from $\mathcal{P}$ to $\widehat{N}$. The converse
transition, namely that the arrival times of $\mathcal{P}$ are
$m(-\log L_1)$, $m(-\log L_2),\ldots$ is secured by our assumption
that $m$ is strictly increasing and continuous (this assumption is
not needed to guarantee the direct transition).

\noindent (a) Write
\begin{eqnarray*}
N(t)-\widehat N(t)&=&(\widehat N(t-\log L)-\widehat
N(t))\1_{\{L\leq 1\}}- (\widehat N(t)-\widehat N(t-\log
L))\1_{\{L>1\}}=:N_1(t)-N_2(t)
\end{eqnarray*}
and observe that
\begin{eqnarray}
N_1(t)&\leq& (\widehat N(t-\log L_1)-\widehat N(t))\1_{\{L_1\leq
1\}}\leq (1+\mathcal{P}^\ast(m(t-\log L_1)-m(-\log
L_1))\notag\\&-&\mathcal{P}^\ast(m(t)-m(-\log L_1)))\1_{\{L_1\leq
1\}},\label{inter222}
\end{eqnarray}
where $\mathcal{P}^\ast:=(\mathcal{P}^\ast(t))_{t\geq 0}$ is a
homogeneous Poisson process of unit intensity which is independent
of $L_1$. More precisely, the arrival times of $\mathcal{P}^\ast$
are $m(-\log L_2)-m(-\log L_1)$, $m(-\log L_3)-m(-\log
L_1),\ldots$. For later use we note that
\begin{eqnarray}\label{inter223}
&&((\mathcal{P}^\ast(m(t-\log L_1)-m(-\log
L_1))-\mathcal{P}^\ast(m(t)-m(-\log L_1)))\1_{\{e^{-t}\leq L_1\leq
1\}})_{t\geq 0}\notag\\&~\od~&((\mathcal{P}^\ast(m(t-\log
L_1))-\mathcal{P}^\ast(m(t)))\1_{\{e^{-t}\leq L_1\leq 1\}})_{t\geq
0},
\end{eqnarray}
where $\od$ means that the distributions of the processes are the
same. Inequality \eqref{inter222} entails
$$\me N_1(t)\leq (1+\me (m(t-\log L_1)-m(t\vee (-\log
L_1)))\1_{\{L_1\leq 1\}}\leq (1+\me (m(t-\log
L_1)-m(t)))\1_{\{L_1\leq 1\}}.$$ In view of \eqref{ineq} for
$t,x\geq 0$
\begin{eqnarray}
&&m(t+x)-m(t)\notag\\&\leq& c_0
((t+x)^q-t^q)+\alpha_0+\alpha_1(x+t)^{q-r_1}+|\beta_0|+|\beta_1|t^{q-r_2}\leq
c_0(x^q \1_{\{q\in
(0,1]\}}\notag\\&+&q(x^{q-1}+t^{q-1})x\1_{\{q\in
(1,2)\}})+\alpha_0+\alpha_1(x^{q-r_1}+t^{q-r_1})+|\beta_0|+|\beta_1|t^{q-r_2}.\label{inter2121}
\end{eqnarray}
We have used subadditivity of $x\mapsto x^\kappa$ on
$\mr_+:=[0,\infty)$ when $\kappa\in (0,1]$ and the mean value
theorem for differentiable functions to obtain
$(t+x)^\kappa-t^\kappa=\kappa x(t+x)^{\kappa-1}$ when $\kappa>1$.
We infer
\begin{equation}\label{mom}
\me (\log_-L_1)^\alpha<\infty\quad \text{for any}~\alpha>0
\end{equation}
as a consequence of $\int_1^\infty y^{\alpha-1}\mmp\{-\log
L_1>y\}{\rm d}y=\int_1^\infty y^{\alpha-1}e^{-m(y)}{\rm
d}y<\infty$, where the finiteness is justified by \eqref{ineq}.
Hence,
\begin{eqnarray*}
\me N_1(t)&\leq& 1+c_0(\me (\log_-L_1)^q \1_{\{q\in (0,1]\}}+q(\me
(\log_-L_1)^q +t^{q-1}\me(\log_-L_1)\1_{\{q\in
(1,2)\}})+\alpha_0\\&+&\alpha_1(\me
(\log_-L_1)^{q-r_1}+t^{q-r_1})+|\beta_0|+|\beta_1|t^{q-r_2}.
\end{eqnarray*}
Thus, the right-hand inequality in part (a) holds with
$r_3=r_1\wedge r_2$ when $q\in (0,1]$ and $r_3=r_1\wedge r_2\wedge
1$ when $q\in (1,2)$.

To analyse $N_2(t)$, set $\theta:=q$ if $q\in (0,1]$ and
$\theta:=q/(2-q)$ if $q\in (1,2)$ and then pick $\varepsilon>0$
such that $\theta+\varepsilon\leq 2q$ when $q\in (0,3/2)$ and take
the same $\varepsilon$ as in \eqref{logLevy} when $q\in [3/2, 2)$.
Further, choose $\delta<1-(q\vee 1)/2$ and $\varrho_1>1$
sufficiently close to one to ensure that
$r_5:=(\theta+\varepsilon)\delta/\varrho_1>q/2$. Put
$\varrho_2:=\varrho_1/(\varrho_1-1)$. It holds that
\begin{eqnarray*}
N_2(t)&=&(\widehat N(t)-\widehat N(t-\log L))\1_{\{1<L\leq
\exp(t^\delta)\}}+(\widehat N(t)-\widehat N(t-\log
L))\1_{\{L>\exp(t^\delta)\}}\\&\leq& (\widehat N(t)-\widehat
N(t-t^\delta))+\widehat N(t)\1_{\{L>\exp(t^\delta)\}}.
\end{eqnarray*}
Condition \eqref{logLevy} ensures that $\me (\log_+
L)^{\theta+\varepsilon}<\infty$ by Theorem 25.3 in
\cite{Sato:1999}. A combination of H\"{o}lder's and Markov's
inequalities yields
\begin{eqnarray*}
\me \widehat N(t)\1_{\{L>\exp(t^\delta)\}}&\leq& (\me (\widehat
N(t))^{\varrho_2})^{1/\varrho_2}(\mmp\{\log
L>t^\delta\})^{1/\varrho_1}\\&\leq& (\me (\widehat
N(t))^{\varrho_2})^{1/\varrho_2}(\me (\log_+
L)^{\theta+\varepsilon})^{1/\varrho_1}
t^{-(\theta+\varepsilon)\delta/\varrho_1}.
\end{eqnarray*}
Since $\widehat N(t)$ has a Poisson distribution of mean $m(t)$,
and $m(t)$ satisfies \eqref{asym}, the right-hand side does not
exceed $\alpha_5+\alpha_4 t^{q-r_5}$ for $t\geq 0$ and some
$\alpha_4, \alpha_5>0$.

Further, using \eqref{ineq} we obtain for $t\geq 0$
\begin{eqnarray*}
\me (\widehat N(t)-\widehat
N(t-t^\delta))&=&m(t)-m(t-t^\delta)\leq c_0 (t^{\delta
q}\1_{\{q\in (0,1]\}}+qt^{q-1+\delta}\1_{\{q\in
(1,2)\}})+\alpha_0+\alpha_1
t^{q-r_1}\\&+&|\beta_0|+|\beta_1|t^{q-r_2}\leq \alpha_7+ \alpha_6
t^{q-r_6}.
\end{eqnarray*}
By our choice of $\delta$, $r_6$ satisfies $r_6>q/2$. We have
proved the left-hand inequality in part (a) with $r_4:=r_5\wedge
r_6$.

\noindent (b) Having written
\begin{eqnarray*}
&&\me\sup_{s\in[0,\,t]}(N(s)-V(s))^2\1_{\{L>1\}}\\&\leq&
3\Big(\me\sup_{s\in[0,\,t]}(\mathcal{P}(m(s-\log L))-m(s-\log
L))^2\1_{\{L>1\}}\\&+&\me\sup_{s\in[0,\,t]}(m(s-\log
L)-m(s))^2\1_{\{L>1\}}+\sup_{s\in[0,\,t]}(m(s)-V(s))^2\Big),
\end{eqnarray*}
we intend to show that  each of the three terms on the
right-hand side is $O(t^q)$.

\noindent {\sc 1st summand}. Recall that
$(\mathcal{P}(t)-t)_{t\geq 0}$ is a martingale with respect to the
natural filtration. Using
$$\sup_{s\in[0,\,t]}(\mathcal{P}(m(s-\log L))-m(s-\log
L))^2\1_{\{L>1\}}\leq \sup_{s\in (-\infty,\,
t]}(\mathcal{P}(m(s))-m(s))^2\leq
\sup_{s\in[0,\,m(t)]}(\mathcal{P}(s)-s)^2$$ and then invoking
Doob's inequality we obtain
\begin{eqnarray*}
\me\sup_{s\in[0,\,t]}(\mathcal{P}(m(s-\log L))-m(s-\log
L))^2\1_{\{L>1\}}&\leq&
\me\sup_{s\in[0,\,m(t)]}(\mathcal{P}(s)-s)^2\\&\leq& 4\me
(\mathcal{P}(m(t))-m(t))^2= 4m(t)=O(t^q).
\end{eqnarray*}

\noindent {\sc 2nd summand}. The following inequalities hold
\begin{eqnarray}
&&\me \sup_{s\in[0,\,t]}(m(s)-m(s-\log
L))^2\1_{\{L>1\}}\notag\\&\leq& (m(t)-m(0))^2\mmp\{\log L>t\}+\me
\sup_{s\in [0,\,t-\log L]}(m(s+\log L)-m(s))^2\1_{\{0<\log L\leq
t\}}\notag\\&\leq& (m(t)-m(0))^2\mmp\{\log L>t\}+\me \sup_{s\in
[0,\,t]}(m(s+\log L)-m(s))^2\1_{\{\log L>0\}}.\label{inter23}
\end{eqnarray}
Note that \eqref{logLevy} entails $\me (\log_+L)^{2q}<\infty$. In
view of \eqref{asym} the first summand on the right-hand side of
\eqref{inter23} is $O(1)$ by Markov's inequality. Using
\eqref{inter2121} in combination with $\me (\log_+L)^{2q}<\infty$
we conclude that the second summand on the right-hand side of
\eqref{inter23} is $O(t^q)$.

\noindent {\sc 3rd summand}. Using \eqref{ineq3} and \eqref{ineq}
yields $\sup_{s\in[0,\,t]}(m(s)-V(s))^2\leq \sup_{s\in
[0,\,t]}(C_1+C_2s^{q-r})^2=O(t^{2q-2r})$ for appropriate constants
$C_1$, $C_2$ and $q/2<r\leq q$. The latter inequality ensures that
$\sup_{s\in[0,\,t]}(m(s)-V(s))^2=O(t^q)$.

To deal with the expectation in question on the event $\{L\leq
1\}$ we write
\begin{eqnarray*}
&&\me\sup_{s\in[0,\,t]}(N(s)-V(s))^2\1_{\{L\leq 1\}}\\&\leq&
3\Big(\me\sup_{s\in[0,\,t]}(\mathcal{P}(m(s-\log
L))-\mathcal{P}(m(s)))^2\1_{\{L\leq
1\}}\\&+&\me\sup_{s\in[0,\,t]}(\mathcal{P}(m(s))-m(s))^2+\sup_{s\in[0,\,t]}(m(s)-V(s))^2\Big).
\end{eqnarray*}
We already know from the previous part of the proof, that the
second and the third summand on the right-hand side are $O(t^q)$.
As for the first summand, we use \eqref{inter222} and
\eqref{inter223} to obtain
\begin{eqnarray*}
&&\me \sup_{s\in[0,\,t]}(\mathcal{P}(m(s-\log
L))-\mathcal{P}(m(s)))^2\1_{\{L\leq 1\}}\\&\leq& \me
(1+\mathcal{P}^\ast(m(t-\log L_1))-m(-\log L_1))^2\1_{\{-\log
L_1>t\}}\\&+&\me \sup_{s\in[-\log
L_1,\,t]}(1+\mathcal{P}^\ast(m(s-\log
L_1))-\mathcal{P}(m(s)))^2\1_{\{0\leq -\log L_1\leq t\}}
\end{eqnarray*}
The principal asymptotic term of the first summand is $\me
(m(t-\log L_1)-m(t))^2\1_{\{-\log L_1>t\}}$. Invoking
\eqref{inter2121} and \eqref{mom} we infer that the last
expression is $o(1)$. To estimate the second summand we write
\begin{eqnarray*}
&&\me \sup_{s\in[-\log L_1,\,t]}(\mathcal{P}^\ast(m(s-\log
L_1))-\mathcal{P}(m(s)))^2\1_{\{0\leq -\log L_1\leq t\}}\\&\leq&
3\Big(\me\sup_{s\in[-\log L_1,\,t]}(\mathcal{P}^\ast (m(s-\log
L_1))-m(s-\log L_1))^2\1_{\{0\leq -\log L_1\leq
t\}}\\&+&\me\sup_{s\in[0,\,t]}(\mathcal{P}^\ast(m(s))-m(s))^2+\me
\sup_{s\in[0,\,t]}(m(s-\log L_1)-m(s))^2\1_{\{-\log L_1\geq
0\}}\Big)\\&\leq& 3\Big(2\me\sup_{s\in[0,\,2t]}(\mathcal{P}^\ast
(m(s))-m(s))^2+\me \sup_{s\in[0,\,t]}(m(s-\log
L_1)-m(s))^2\1_{\{-\log L_1\geq 0\}}\Big).
\end{eqnarray*}
The last expression is $O(t^q)$ which can be seen by mimicking the
arguments used in the previous part of the proof.

\noindent (c) A specialization of the functional limit theorem for
the renewal processes with finite variance (see, for instance,
Theorem 3.1 on p.~162 in \cite{Gut:2009}) yields
\begin{equation}\label{intermed}
\frac{\mathcal{P}(t\cdot)-(t\cdot)}{t^{1/2}}~\Rightarrow~
B(\cdot),\quad t\to\infty
\end{equation}
in the $J_1$-topology on $D$.

It is well-known (see, for instance, Lemma 2.3 on p.~159 in
\cite{Gut:2009}) that the composition mapping $(x, \varphi)\mapsto
(x\circ \varphi)$ is continuous at continuous functions $x:\mr_+
\to \mr$ and nondecreasing continuous functions $\varphi: \mr_+\to
\mr_+$. This observation taken together with \eqref{intermed} and
\eqref{dini} enables us to conclude that
\begin{equation}\label{intermed2}
\frac{\widehat N(t\cdot)-m(t\cdot)}{(c_0 t^q)^{1/2}}~\Rightarrow~
W(\cdot),\quad t\to\infty
\end{equation}
in the $J_1$-topology on $D$. Further, we have for all $s_0, t>0$
$$\sup_{s\in [0,\, s_0]}|m(ts)-c_0 ts|\leq \max (\alpha_0+\alpha_1
(ts_0)^{q-r_1}, |\beta_0|+|\beta_1|(ts_0)^{q-r_2})$$ whence
$$\lim_{t\to\infty} t^{-q/2}\sup_{s\in [0,\, s_0]}|m(ts)-c_0 ts|=0,$$ where the assumption
$r_1,r_2>q/2$ has to be recalled. This enables us to replace in
\eqref{intermed2} $m(t\cdot)$ with $c_0 (t\cdot)^q$ thereby giving
\begin{equation}\label{fclt}
\frac{\widehat N(t\cdot)-c_0 (t\cdot)^q }{(c_0
t^q)^{1/2}}~\Rightarrow~ W(\cdot),\quad t\to\infty
\end{equation}
in the $J_1$-topology on $D$.

By Skorohod's theorem it is possible to define $\widetilde N$ and
$W$ on the same probability space, so that limit relation
\eqref{fclt} holds in the $J_1$-topology on $D$ almost surely,
hence also for $t-\log L$ in place of $t$. Now, in the centering
and normalization functions for $\widetilde N(t-\log L)$ we wish
to replace $t-\log L$ by $t$. This is trivial as far as the
normalization is concerned, for $(t-\log L)^{q/2}\sim t^{q/2}$
a.s.\ as $t\to\infty$. As for the centering, it suffices to show
that, for all $s_0>0$,
\begin{equation}\label{cent}
t^{-q/2}\sup_{s\in [0,\,s_0]}|(ts)^q-((t-\log L)s)^q|=t^{-q/2}
s_0^q|t^q-(t-\log L)^q| ~\tp~0,\quad t\to\infty.
\end{equation}
If $q\in (0,1]$, then $|t^q-(t-\log L)^q|\leq |\log L|^q$ by
subadditivity of $x\mapsto x^q$ on $\mr_+$. If $q\in (1,2)$, then
$|t^q-(t-\log L)^q|\leq q|\log L|(t+|\log L|)^{q-1}$ by the mean
value theorem for differentiable functions. These inequalities
entail \eqref{cent} which completes the proof of part (c) and the
lemma.
\end{proof}



\end{document}